\documentclass[11pt,twoside]{article}
\usepackage{artmods,macros}
\usepackage{amsmath,amssymb}
\usepackage{showlabels}
\usepackage{enumerate}
\usepackage{color}
\usepackage{caption}
\usepackage{graphicx}
\usepackage{subfig}
\usepackage{algorithm}
\usepackage{algpseudocode}
\usepackage{rotating}
\usepackage{tikz}

\title{QUASI-NEWTON METHODS FOR MINIMIZING A QUADRATIC FUNCTION
  SUBJECT TO UNCERTAINTY}
\author{Shen PENG\footnotemark\addtocounter{footnote}{-1}
   \and Gianpiero CANESSA\footnotemark\addtocounter{footnote}{-1}
   \and David EK\footnotemark\addtocounter{footnote}{-1}
   \and Anders FORSGREN\thanks{\footAF}}

\def\footAF{Optimization and Systems Theory, Department of
            Mathematics, KTH Royal Institute of Technology, SE-100 44
            Stockholm, Sweden ({\tt
              \{shenp,canessa,daviek,andersf\}@kth.se}).}

\input macros

\makeatletter
\def\calK{{\mathcal K}}

\def\AFcomment#1{{\color{red}\emph{AF: #1}}}

\def\AFtoGCSP#1{\emph{AF$\rightarrow$GC,SP: #1}}
\def\AFtoDE#1{\emph{AF$\rightarrow$DE: #1}}
\makeatother

\newtheorem{remark}{Remark}

\algnewcommand\AND{\textbf{and} }

\begin{document}

\maketitle\thispagestyle{plain}

\begin{abstract}
  We investigate quasi-Newton methods for minimizing a strictly convex
  quad\-ra\-tic function which is subject to errors in the evaluation
  of the gradients. The methods all give identical behavior in exact
  arithmetic, generating minimizers of Krylov subspaces of increasing
  dimensions, thereby having finite termination. A BFGS quasi-Newton
  method is empirically known to behave very well on a quadratic
  problem subject to small errors. We also investigate large-error
  scenarios, in which the expected behavior is not so clear. In
  particular, we are interested in the behavior of quasi-Newton
  matrices that differ from the identity by a low-rank matrix, such as
  a memoryless BFGS method. Our numerical results indicate that for
  large errors, a memory-less quasi-Newton method often outperforms a
  BFGS method. We also consider a more advanced model for generating
  search directions, based on solving a chance-constrained
  optimization problem. Our results indicate that such a model often
  gives a slight advantage in final accuracy, although the
  computational cost is significantly higher.
\end{abstract}

\textbf{Keywords:} Quadratic programming, quasi-Newton method, stochastic quasi-Newton method, chance constrained model  \newline

\section{Introduction}\label{sec-intro}

A strictly convex $n$-dimensional quadratic function may be written on
the form
\[
q(x) = \half x\T H x + c\T x + d,
\]
where $H$ is a positive definite and symmetric $n\times n$-matrix,
$c$ is an $n$-dimensional vector and $d$ is a constant. The
optimization problem of minimizing $q(x)$ is equivalent to solving
$\grad q(x)=0$, i.e., solving the linear equation $Hx+c=0$.

One way to do so by a direct method is to find an initial point $x_0$
and associated gradient $g_0=Hx_0+c$. Then generate $x_k$ and $g_k$,
with $g_k = H x_k +c$, such that $x_k$ is the minimizer of $q(x)$ on
$x_0 + \calK_k(g_0,H)$, where
\[
\calK_0(g_0,H)=\{0\}, \quad \calK_k(g_0,H)=\Span\{g_0,Hg_0,H^2 g_0,\dots,H^{k-1}
g_0\}, \quad k=1,2,\dots.
\]
This is equivalent to $g_k$ being orthogonal to $\calK_k(g_0,H)$, so
that $g_0$, $g_1$, \dots, $g_k$ form an orthogonal basis for
$\calK_{k+1}(g_0,H)$. Since there can be at most $n$ nonzero
orthogonal vectors, there is an $r$, with $r\le n$, such that
$g_r=0$. Consequently, $x_r$ is the minimizer of $q(x)$.

A method for computing $x_1$, $x_2$, \dots, $x_r$ this way is
characterized by the search direction $p_k$ leading from $x_k$ to
$x_{k+1}$. Given $p_k$, the step length $\alpha_k$ is given by
minimizing $q(x_k+\alpha p_k)$, i.e.,
\begin{equation}\label{eqn-exactlinesearch}
\alpha_k = - \frac {g_k^T p_k}{p_k^T H p_k}.
\end{equation}
Therefore, it suffices to characterize $p_k$. One characterization is
given by the conditions that
\begin{subequations} \label{eqn-pkcorrect}
\begin{eqnarray}
& (i) \quad & \mbox{$p_k$ is a linear combination of $g_0$,
$g_1$, $\dots$, $g_k$, in addition to} \\
& (ii) \quad & \mbox{satisfying $g_i^T p_k = -g_k^T g_k$, $i=0,\dots,k,$}
\end{eqnarray}
\end{subequations}
see, e.g., \cite[Lemma 1]{EF21}. There could be an arbitrary scaling $g_i^T
p_k=c_k$ for a nonzero scalar $c_k$. Throughout, we will use
$c_k=-g_k^T g_k$.

The method of conjugate gradients gives a short recursion for the
search direction $p_k$ satisfying (\ref{eqn-pkcorrect}). It may be
written on the form
\begin{equation}\label{eqn-pkCG}
p_k =
\begin{cases}
-g_0, & k = 0, \\
-g_{k} + \frac{g_k^T g_k}{g_{k-1}^T g_{k-1}\drop} p_{k-1}, &
k=1,2,\dots,r-1.
\end{cases}
\end{equation}
See, e.g., \cite[Chapter 6]{Saa95} for an introduction to the method
of conjugate gradients.

An alternative way of computing the search direction $p_k$ satisfying
(\ref{eqn-pkcorrect}) is through a quasi-Newton method, in which $p_k$
is defined by a linear equation $B_k p_k=-g_k$, for $B_k$ a symmetric
positive-definite matrix. A well-known method for which $B_k$ gives
$p_k$ satisfying (\ref{eqn-pkcorrect}) is the BFGS quasi-Newton
method, in which $B_0=I$, and $B_{k}$ is formed by adding a symmetric
rank-2 matrix to $B_{k-1}$. In our setting, the BFGS method may be
viewed as the ``ideal'' update that dynamically transforms $B_k$ from
$I$ to $H$ in $n$ steps. Identity curvature is transformed to $H$
curvature in one dimension at each step, and this information is
maintained throughout. The precise formulas will be given in
Section~\ref{sec-Bk}.

In exact arithmetic, the method of conjugate gradients and a BFGS
quasi-Newton method compute identical iterates in the setting we
consider, minimizing a strictly convex quadratic problem using exact
linesearch. In this situation, the recursion formula for the method of
conjugate gradients is to prefer, since the computational cost for
solving with $B_k$ for the BFGS method increases with $k$.

In finite precision arithmetic, the BFGS quasi-Newton method may still
be expected to compute search directions of very high quality as the
Hessian is approximated on subspaces of higher dimension. This is not
to be expected for the method of conjugate gradients. Our interest is
to study the situation where noise is added to the gradients, thereby
considering a situation with significantly higher level of error than
finite precision arithmetic. In this situation, it is not clear that
the BFGS quasi-Newton method is superior, in the sense that the
Hessian approximation may become inaccurate. It is here also of
interest to consider the method of steepest descent, where the search
direction is the negative of the gradient. For exact arithmetic, the
search directions satisfying (\ref{eqn-pkcorrect}) will outperform
steepest descent due to the property of minimizing the quadratic
objective function over expanding subspaces. In the case of large
noise, this is not at all clear.

The reason for noise in the gradients can be seen in different
perspectives.  Firstly, as mentioned above, the finite precision
arithmetic gives a residual between the evaluated gradients and the
true gradients.  Secondly, in many practical problem, such as
PDE-constrained optimization, the objective function often contains
computational noise created by an inexact linear system solver,
adaptive grids, or other internal computations.  In other cases, the
noise in the gradients can be due to stochastic errors. For example,
when minimizing $q(x) = \mathbb{E}[f(x;\xi)]$, where $\xi$ is a random
variable. With given sample set $\Xi=\{\xi^i, i=1,\cdots,N \}$,
instead of $q(x)$, the following empirical expectation will be
considered:
\[
      \tilde q(x) = \frac{1}{N}\sum_{\xi^i \in \Xi}f(x;\xi^i).
\]
Due to the randomness of samples, $\tilde q(x) = q(x) + \varepsilon$, and $\nabla \tilde q(x) = \nabla q(x) + \epsilon$, where $\varepsilon, \epsilon$ are random noise.

The basis for the methods we consider is that they compute search
directions identical to those of the BFGS method and the method of
conjugate gradients, i.e., satisfying (\ref{eqn-pkcorrect}), in exact
arithmetic. In particular, we are interested in a setting where $B_k$
is the identity matrix plus a symmetric matrix of rank two. We will
refer to such a quasi-Newton matrix as a \emph{low-rank} quasi-Newton
matrix and name the corresponding method a low-rank quasi-Newton
method. Our intention is to investigate the behavior of a low-rank
quasi-Newton method compared to the method of steepest descent,
thereby mimicking a situation where two more vectors in addition to
$g_k$ are used at iteration $k$. The corresponding search direction
can then be computed from a two-by-two system. For comparison, we also
compare to the BFGS quasi-Newton method. Our choice of quadratic
problem allows an environment where the behavior of the methods in
infinite precision is known, and we can study the effect of noise.

In addition, we investigate the potential for improving performance of
the quasi-Newton method by formulating robust optimization problems of
chance-constraint type for computing the search directions. These
methods become of higher interest in the case of large noise and
multiple copies of the gradients. Our interest is to capture the
essence of the behavior, and try to understand the interplay between
quality in computed direction compared to robustness given by the
chance constraints. The computational cost will always be
significantly higher, but our interest is to see if we can gain in
terms of robustness and accuracy of the computed solution.

\subsection{Background and related work}

The paper builds on previous work in the setting of exact arithmetic
and finite precision arithmetic. Forsgren and Odland~\cite{FO18} have
studied exact linesearch quasi-Newton methods for minimizing a
strictly convex quadratic function, and given necessary and sufficient
conditions for a quasi-Newton matrix $B_k$ to generate a search
direction which is parallel to that of (\ref{eqn-pkcorrect}) in exact
arithmetic.
With exact linesearch methods, Ek and Forsgren~\cite{EF21} have
studied certain limited-memory quasi-Newton Hessian approximations for
minimizing a convex quadratic function in the setting of finite
precision arithmetic. Dennis and Walker~\cite{DW84} have considered
the use of bounded-deterioration quasi-Newton methods implemented in
floating-point arithmetic where only inaccurate values are
available. In contrast, our work allows for large noise and we study
performance on a set of test problems.



In the present manuscript, we consider a situation where the function
values and gradients cannot be easily obtained and only noisy
information about the gradient is available.  To handle this
situation, some stochastic methods are proposed to minimize the
objective function with inaccurate information. Our setting is
minimizing a strictly convex quadratic function.

For strongly convex problems, Mokhtari and Ribeiro~\cite{MR13} have
proposed a regularized stochastic BFGS method and analyzed its
convergence, and Mokhtari and Ribeiro~\cite{MR15} have further studied
an online L-BFGS method.  Berahas, Nocedal and Takac~\cite{BNT16} have
considered the stable performance of quasi-Newton updating in the
multi-batch setting, illustrated the behavior of the algorithm and
studied its convergence properties for both the convex and nonconvex
cases.  Byrd~\etal~\cite{BHNS16} have proposed a stochastic
quasi-Newton method in limited memory form through subsampled
Hessian-vector products.  Shi~\etal~\cite{shi2020noisetolerant}
have proposed practical extensions of the BFGS and L-BFGS methods for
nonlinear optimization that are capable of dealing with noise by
employing a new linesearch technique.  More recently, Xie, Byrd and
Nocedal~\cite{XBN20} have considered the convergence analysis of
quasi-Newton methods when there are (bounded) errors in both function
and gradient evaluations, and established conditions under which an
Armijo-Wolfe linesearch on the noisy function yields sufficient
decrease in the true objective function.

Unlike the stochastic quasi-Newton methods, which are based on the
subsampled gradients or Hessians, there are also other stochastic
tools to reduce the effect of noise when generating the search
direction. Lucchi~\etal~\cite{LMH15} have studied quasi-Newton method
by incorporating variance reduction technique to reduce the effect of
noise in Hessian matrices by proposing a variance-reduced stochastic
Newton method.  This method keeps the variance under control in the
use of a multi-stage scheme.  Moritz~\etal~\cite{MNJ16} have proposed a
linearly convergent method that integrates the L-BFGS method to
alleviate the effect of noisy gradients with the variance reduction
technique by adding the residual between subsample gradient and full
gradient to the noisy gradient.

In addition, chance constraint is a natural approach to handle the
effect of random noise \cite{Ahm18}.  Therefore, chance constraint has
the potential to reduce the effect of noise when generating the search
direction.  By integrating chance constraints in the design of
quasi-Newton methods, we investigate the ability to improve robustness
into the computation of the search direction in the presence of noise.



\section{Suggestions on quasi-Newton matrices}\label{sec-Bk}

As mentioned in Section~\ref{sec-intro}, a well-known method for which
$B_k$ gives $p_k$ satisfying (\ref{eqn-pkcorrect}) is the BFGS
quasi-Newton method. In the BFGS method, $B_0=I$ and
\begin{align}
B_k & = B_{k-1} + \frac1{g_{k-1}^T p_{k-1}}
 g_{k-1}g_{k-1}^T  \nonumber \\
& +
\frac1{\alpha_{k-1}(g_k-g_{k-1})^Tp_{k-1}}(g_{k}-g_{k-1})(g_{k}-g_{k-1})^T,
\quad k=1,\dots,r.
\label{eqn-BFGS}
\end{align}
Expansion gives
\begin{align} \nonumber
B_k & = B_{k-1} + \frac1{g_{k-1}^T p_{k-1}}
 g_{k-1}g_{k-1}^T
 +
\frac1{\alpha_{k-1}(g_k-g_{k-1})^Tp_{k-1}}(g_{k}-g_{k-1})(g_{k}-g_{k-1})^T \\
& = I + \sum_{i=0}^{k-1} \frac1{g_i\T p_i\drop} g_i g_i^T +
\sum_{i=0}^{k-1} \frac1{\alpha_l(g_{i+1}-g_i)^T p_i}
(g_{i+1}-g_{i})(g_{i+1}-g_{i})^T.
\label{eqn-BFGSfull}
\end{align}
For the quadratic case with exact linesearch, the BFGS matrix of
(\ref{eqn-BFGSfull}) takes the form
\begin{align}\nonumber
B_k & = B_{k-1} - \frac1{g_{k-1}^T g_{k-1}\drop} g_{k-1}\drop g_{k-1}^T +
\frac1{p_{k-1}^T H p_{k-1}\drop} H p_{k-1}\drop p_{k-1}^T H \\
& = I - \sum_{i=0}^{k-1} \frac1{g_i\T g_i\drop} g_i g_i^T +
\sum_{i=0}^{k-1} \frac1{p_i^T H p_i\drop} H p_i\drop p_i^T H,
\label{eqn-BFGSfullquadratic}
\end{align}
see, e.g., \cite{EF21}. If $n$ steps are taken, then with $P_n = ( p_0
\ p_1 \ \cdots \ p_{n-1})$, it follows that $P_n$ is square and
nonsingular, so that
\begin{align}
H & = H H\inv H =
H P_n P_n\inv H\inv P_n\Tinv P_n^T H  \nonumber \\
& =
H P_n (P_n^T H P_n)\inv P_n^T H =
\sum_{i=0}^{n-1} \frac1{p_i^T H p_i\drop} H p_i\drop p_i^T H,
\label{eqn-H}
\end{align}
where the conjugacy of the $p_i$s is a consequence of
(\ref{eqn-pkcorrect}). In this situation,
(\ref{eqn-BFGSfullquadratic}) may therefore be seen as a dynamic way
of generating the true Hessian in $n$ steps, if the method does not
converge early, as $B_n=H$ by
\[
I - \sum_{i=0}^{n-1} \frac1{g_i\T g_i\drop} g_i\drop g_i^T = 0 \text{and}
\sum_{i=0}^{n-1} \frac1{p_i^T H p_i\drop} H p_i\drop p_i^T H = H.
\]
This is a consequence of the orthogonal gradients then spanning the
whole space in combination with (\ref{eqn-H}). Consequently, the BFGS
method may be viewed as the ``ideal'' update that dynamically
transforms $B_k$ from $I$ to $H$ in $n$ steps. Identity curvature is
transformed to $H$ curvature in one dimension at each step, and this
curvature information is maintained throughout.

The discussion above may be generalized to a general class
of quasi-Newton matrices $B_k$ of the form
\begin{equation}\label{eqn-Bklimmemory}
B_k = V_k +
\sum_{i=0}^{k-1} \rho_i (g_{i+1}-g_i)(g_{i+1}-g_i)^T,
\end{equation}
where $V_k p_k=-g_k$ for $p_k$ satisfying (\ref{eqn-pkcorrect}), and
$\rho_i$, $i=0,\dots,k-1$, are nonnegative scalars. In exact
arithmetic and under exact linesearch, the specific values of
$\rho_i$, $i=0,\dots,k-1$, have no impact on the search direction, due
to (\ref{eqn-pkcorrect}). In a noisy framework, they will make a
difference, and we will pay attention to how they are selected.

As discussed in Section~\ref{sec-intro}, we are particularly
interested in low-rank quasi-Newton matrices. We will therefore
consider a \emph{memoryless BFGS} quasi-Newton method, in which
\begin{equation}\label{eqn-Vkmemoryless}
V_k  = I - \frac1{p_{k-1}^T p_{k-1}\drop} p_{k-1} p_{k-1}^T,
\end{equation}
in addition to $\rho_i=0$, $i=0,\dots,k-2$, so that
\begin{equation}\label{eqn-Bkmemoryless}
B_k = I - \frac1{p_{k-1}^T p_{k-1}\drop} p_{k-1} p_{k-1}^T
+ \rho_{k-1} ( g_k - g_{k-1} ) ( g_k - g_{k-1})^T,
\end{equation}
where the value of $\rho_{k-1}$ is given by the \emph{secant
  condition} $\alpha_{k-1}B_k p_{k-1}= g_k-g_{k-1}$. We denote this
particular value of $\rho_{k-1}$ by $\hat\rho_{k-1}$. For exact
linesearch, $B_k$ of (\ref{eqn-Bkmemoryless}) gives
\begin{equation}\label{eqn-secant}
\hat \rho_{k-1} = -\frac1{\alpha_{k-1}g_{k-1}^Tp_{k-1}\drop}.
\end{equation}
Then $V_k p_k=-g_k$ for $p_k$ satisfying (\ref{eqn-pkcorrect}) in the case of
exact arithmetic, see, e.g., \cite[Proposition 1]{FO18}. The $V_k$ of
the memoryless BFGS matrix given by (\ref{eqn-Vkmemoryless}) is
analogous to the first two terms of the BFGS matrix of
(\ref{eqn-BFGSfullquadratic}), but the matrix is singular with its
nullspace restricted to the one-dimensional span of $p_{k-1}$, as
opposed to the span of all previous gradients. In addition, the
memoryless BFGS matrix $B_k$ of (\ref{eqn-Bkmemoryless}) is nonsingular
as $\rho_{k-1}>0$ and $( g_k - g_{k-1})^T p_{k-1} \ne 0$.

We will also interpret the method of conjugate gradients in a
quasi-Newton framework, by forming the \emph{symmetric CG}
quasi-Newton matrix $B_k$ given by
\begin{equation}\label{eqn-BksymCG}
B_k = \left( I - \frac{1}{g_{k-1}^T p_{k-1}\drop}g_k
    p_{k-1}^T\right)
\left( I - \frac{1}{g_{k-1}^T p_{k-1}\drop} p_{k-1}\drop g_k^T\right).
\end{equation}
This matrix is formed by rewriting the recursion (\ref{eqn-pkCG}) and
making an additional symmetrization, see~\cite{FO18}. It can be put in
the matrix family given by (\ref{eqn-Bklimmemory}) by setting
$V_k=B_k$ and $\rho_i=0$, $i=0,\dots,k-1$.

In summary, we will consider quasi-Newton matrices of the form
(\ref{eqn-Bklimmemory}). In particular, we will consider two specific
low-rank matrices. The quasi-Newton matrices $B_k$ given by
\emph{memoryless BFGS} in (\ref{eqn-Bkmemoryless}) and \emph{symmetric
  CG} in (\ref{eqn-BksymCG}) are both low-rank quasi-Newton matrices
that differ from the identity by a symmetric rank two matrix and
fulfill the conditions we require, (\ref{eqn-pkcorrect}). They will be
used in our computational study.

\section{A chance-constrained model for finding the search
  direction}\label{sec-model}

In addition to investigating the behavior of the quasi-Newton methods
discussed so far, we are also interested in investigating the
potential of increasing the performance of the quasi-Newton methods in
the presence of noise by selecting parameters from a
chance-constrained optimization problem.

The aim is to design a quasi-Newton matrix $B_k$, with $B_k\succ 0$,
so as to compute a search direction of high quality.  For the case of
exact arithmetic, i.e., no noise, our model direction is the direction
$p_k$ that satisfies (\ref{eqn-pkcorrect}). The interest is now to
investigate and design quasi-Newton matrices in the presence of
noise. In particular, we are interested in studying the performance of
different methods for different noise levels. For a given quasi-Newton
matrix $B_k$, the search direction $p_k$ is computed from $B_k p_k =
-g_k$.

  As there exists noise in each iteration, it means that the obtained
  gradient is not accurate, it is the combination of the true
  gradient and some white noise.  The update of search direction may
  result in a non-descent direction because of the influence by the
  noise.  Then, we have the following proposition to show that the
  search direction satisfying $B_k p_k = -g_k$ is a descent direction
  under certain conditions even with noise in each iteration.

\begin{proposition}\label{Th_noise}
  Consider iteration $k$ of a quasi-Newton method for minimizing
  $q(x)$. Let $g_k = \bar{g}_k +
\epsilon$, where $\bar{g}_k$ is the
  true gradient and $\epsilon$ is the noise generated with mean equal
  to $0$.  If $B_k \succ 0$ and $\|\epsilon\| <
  \frac{1}{\|B_k^{-1}g_k\|}g_k^TB_k^{-1}g_k$, the direction $p_k$,
  satisfying $B_k p_k = -g_k$, is a descent direction at point $x_k$.
\end{proposition}
\begin{proof}
  The direction $p_k$ is a descent direction if $\bar{g}_k^Tp_k < 0$.
  Since $B_k p_k = -g_k$, we have
  \[
  \bar{g}_k^Tp_k = (g_k - \epsilon)^Tp_k = -(g_k - \epsilon)^TB_k^{-1}g_k
  = -g_k^TB_k^{-1}g_k + \epsilon^TB_k^{-1}g_k.
  \]
  Hence, $p_k$ is a descent direction if $ g_k^TB_k^{-1}g_k - \epsilon^TB_k^{-1}g_k > 0 $.

  As $ g_k^TB_k^{-1}g_k - \epsilon^TB_k^{-1}g_k \geq g_k^TB_k^{-1}g_k - \|\epsilon\|\|B_k^{-1}g_k\| $,
  $ \|\epsilon\| < \frac{1}{\|B_k^{-1}g_k\|}g_k^TB_k^{-1}g_k$
  implies $p_k$ is a descent direction.
  This concludes the proof.
\end{proof}

A consequence is that the property of $p_k$ being a descent direction
may be lost if the termination criteria on $\norm{\tilde{g}_k}$ is set smaller
  than the noise level, as observed in the following remark.

\begin{remark}\label{Remark_m1}
  Proposition \ref{Th_noise} shows that when the noise $\epsilon$ is
  small enough, the direction satisfying $B_k p_k = -g_k$ is a descent
  direction.  However, if the noise $\epsilon$ is large, the direction
  obtained from $B_k p_k = -g_k$ may not be a descent direction.  In
  particular, let $\tau>0$ denote the tolerance level of the stopping
  criterion based on $\norm{g_k}$.  Consider the situation when
  $\|\epsilon\| > \tau$.  If $x_k$ is close to termination solution,
  the value of $\|g_k\|$ is close to $\tau$.  In this case, we could
  have $ \tau < \|g_k\| \leq \|\epsilon\|$.  Since
  $\frac{1}{\|B_k^{-1}g_k\|}g_k^TB_k^{-1}g_k \leq
  \frac{1}{\|B_k^{-1}g_k\|}\|g_k\| \|B_k^{-1}g_k \| = \|g_k\| \leq
  \|\epsilon\|$, the conditions in Proposition \ref{Th_noise} can not
  hold.  Therefore, the direction satisfying $B_k p_k = -g_k$ may be
  not a descent direction anymore.
\end{remark}

In our quasi-Newton setting, the aim is not only to generate descent
directions, but also to generate search directions of high
quality. Suppose $r_k(p_k)$ is a quality measure for the search
direction $p_k$ at iteration $k$.  The aim is to minimize the quality
measure $r_k(p_k)$ such that $ B_k p_k = -g_k, B_k \succ 0$.
Therefore, for a given quality measure $r_k(p_k)$, we can generate a
search direction with highest quality by solving the following
optimization problem:
\begin{equation} \label{form:m0}
\begin{array}{ll}
\minimize{p_k, B_k} & { r_k(p_k)  }\\
 \subject & B_k p_k = -g_k, \\
 & B_k \succ 0.
\end{array}
\end{equation}
In addition, we typically require $B_k$ to have some additional
properties, as problem (\ref{form:m0}) becomes highly complex
otherwise.

The model \eqref{form:m0} is actually a deterministic model, where the
noisy gradients are deterministic.
However, as mentioned in Section \ref{sec-intro}, in some practical problems, the
gradients themselves are random because of the randomness in
the original objective quadratic function.  Therefore, it is more
natural to view the gradients as random vectors in model
\eqref{form:m0}.
At iteration $k$, it is assumed that $\tilde{g}_k = \bar{g}_k +
\tilde{\epsilon}$, where $\tilde{\epsilon}$ is a random noise with mean equal to $0$.
Hence, considering the randomness and to overcome the
shortcoming of model \eqref{form:m0} as mentioned in Remark
\ref{Remark_m1}, the model \eqref{form:m0} can be formulated as the
following chance constrained model:
\begin{equation}\label{form:m0_CC} 
\begin{array}{lll}
\minimize{t, p_k, B_k} & t   \\
\subject & \mathbb{P}\left\{
  r_k(p_k) \leq t, B_k p_k = -\tilde{g}_k, B_k \succ 0
\right\} \geq 1-\beta,
\end{array}
\end{equation}
where $\beta \in (0,1)$ is a given probability level. The value of $\beta$ indicates the risk-aversion of the decision maker, where 0 is the most conservative approach as we need to comply with the supremum value of the underlying random vector, while higher values would make our solution averse to risk. Even small values of $\beta$ can have significant impact to the results \cite{Bar16}, therefore studying the behaviour of the solution while $\beta$ is 0 and close to 0 (typically 0.01 or 0.05) is the usual approach. The chance
constraint in problem \eqref{form:m0_CC} not only guarantees the
validation of quasi-Newton setting with probability at least
$1-\beta$, but also controls the quality measure of the search direction.

To show that the search direction obtained by solving problem
\eqref{form:m0_CC} is a descent direction with a high probability, we
have the following proposition.

\begin{proposition}\label{Th_noiseCC}
Denote $\tilde{g}_k = \bar{g}_k + \tilde{\epsilon}$, where $\bar{g}_k$ is the true gradient and $\tilde{\epsilon}$ is a random noise with mean equal to $0$.
If $B_k$ is invertible and $\mathbb{P}\left\{\|\tilde{\epsilon}\| <  \frac{1}{\|B_k^{-1}\tilde{g}_k\|}\tilde{g}_k^TB_k^{-1}\tilde{g}_k\right\} \geq 1 - \alpha$, the direction $p_k$ satisfying the chance constraint in problem \eqref{form:m0_CC} is a descent direction at point $x_k$ with probability at least $\max\{1-\alpha -\beta,0\}$.
\end{proposition}
\begin{proof}
  From Proposition \ref{Th_noise}, we have that the direction $p_k$ is a descent direction if the following constraints are satisfied:
   \[
   B_k p_k = -\tilde{g}_k,
   B_k \succ 0.
   \]
Then, we have
\begin{eqnarray*}
   && \mathbb{P}\left\{
\begin{array}{l}
  B_k p_k = -\tilde{g}_k,
  B_k \succ 0,\\
  \|\epsilon\| <  \frac{1}{\|B_k^{-1}\tilde{g}_k\|}\tilde{g}_k^TB_k^{-1}\tilde{g}_k
\end{array}
\right\} \\
    &\geq&\mathbb{P}\left\{
\begin{array}{l}
  r_k(p_k) \leq t, B_k p_k = -\tilde{g}_k, B_k \succ 0,\\
  \|\epsilon\| <  \frac{1}{\|B_k^{-1}\tilde{g}_k\|}\tilde{g}_k^TB_k^{-1}\tilde{g}_k
\end{array}
\right\} \\
   &\geq& \max\Bigg\{\mathbb{P}\left\{
  r_k(p_k) \leq t, B_k p_k = -\tilde{g}_k, B_k \succ 0
\right\}
\\
&& + \mathbb{P}\left\{\|\epsilon\| <  \frac{1}{\|B_k^{-1}\tilde{g}_k\|}\tilde{g}_k^TB_k^{-1}\tilde{g}_k\right\} - 1,0 \Bigg\} \\
&\geq & \max\{1-\alpha-\beta,0\}.
\end{eqnarray*}
The second inequality comes from Fr\'echet inequality.
Therefore, the conclusion can be obtained.
\end{proof}

In contrast to the deterministic case, we can still maintain a descent
direction with positive probability even for $\norm{\tilde{g}_k}< \epsilon$,
as observed in the following remark.

\begin{remark}\label{Remark_m2}
From Proposition \ref{Th_noiseCC}, we can observe that even $x_k$ is close to the optimal solution and the value of $\|\tilde{g}_k\|$ is small, there still exists a constant $0<\bar{\alpha}<1$ such that $\mathbb{P}\left\{\|\epsilon\| <  \frac{1}{\|B_k^{-1}\tilde{g}_k\|}\tilde{g}_k^TB_k^{-1}\tilde{g}_k\right\} \geq 1 - \bar{\alpha}$.
This implies that problem \eqref{form:m0_CC} can always provide a descent direction $p_k$ with some probability no matter how close $x_k$ is to the optimal solution.
\end{remark}

A chance-constrained model is expected to be computationally
expensive. Our interest is to study its behavior, in particular to
see how such an approach might be able to achieve improved accuracy,
as indicated by Remark~\ref{Remark_m2}.

\subsection{Suggestion on quality measure}

With these propositions concerning the search direction, an interesting
issue is how to determine the quality measure of the search direction
$p_k$ when the method is applied to an unconstrained quadratic
optimization problem.  We suggest an approach based on the
characterization given in \eqref{eqn-pkcorrect}. In the exact
arithmetic case, \eqref{eqn-pkcorrect} gives $(g_{i+1}-g_i)^T p_k =
0$, $i=0,\dots,k-1$. Therefore, in this situation, the desired $p_k$
would give a global minimum zero with respect to the measure
\begin{equation}\label{cg-objective}
r_k(p_k) = \sum_{i=1}^{k-1}\left( (g_{i+1}- g_i)^T p_k \right)^2.
\end{equation}
In the exact arithmetic case, $p_k$ is orthogonal to the affine span
of the generated gradients. In the noisy setting, we will use this
measure to show how close the direction is to the characterization in
\eqref{eqn-pkcorrect}.


Motivated by the discussion above, at iteration $k$, for a
given positive semidefinite matrix $V_k$, 
we will consider quasi-Newton matrices in the family given by
(\ref{eqn-Bklimmemory}) and wish to find $\rho$ as the solution of
\begin{equation} \tag{D} \label{form:m1}
\begin{array}{ll}
\minimize{\rho, p_k, B_k\succ 0} & \disp{ \sum_{i=0}^{k-1}
  ((g_{i+1}- g_i)^T p_k )^2  }\\
 \subject & B_k p_k = -g_k, \\
 & B_k = V_k +  \sum_{i=0}^{k-1} \rho_{i} (g_{i+1}-g_i)(g_{i+1}-g_i)^T.
\end{array}
\end{equation}
For the given $V_k$, in exact arithmetic and under exact linesearch,
the specific values of $\rho$ have no impact on the search
direction. However, the specific values may ensure nonsingularity of
$B_k$ and numerical stability. Note also that the sum of rank-one
matrices in (\ref{eqn-Bklimmemory}) is similar to terms present in the
BFGS Hessian approximation of (\ref{eqn-BFGSfull}).

Considering the possible randomness in the gradient $\tilde{g}_k$, the chance constrained model \eqref{form:m0_CC} should take equations \eqref{eqn-Bklimmemory} and \eqref{cg-objective} into consideration.
In addition, from \eqref{eqn-BksymCG}, we can notice that the matrix $V_k$ can be dependent on the noisy gradient $\tilde{g}_k$ in some cases, which implies that $V_k$ is random.
Therefore, in the random case, we denote
\begin{equation}\label{eqn-Bklimmemory_Random}
B_k = \tilde{V}_k +  \sum_{i=0}^{k-2} \rho_{i} (g_{i+1}-g_i)(g_{i+1}-g_i)^T
 +\rho_{k-1} (\tilde{g}_k-g_{k-1})(\tilde{g}_k-g_{k-1})^T,
\end{equation}
where $\tilde{V}_k$ is a random matrix.  Assume that the gradients
$g_i, i=0,\dots, k-2$, have been evaluated, which can be the realized
values of noisy gradients or average values of noisy gradient samples.
Then, based on model \eqref{form:m1}, the chance constrained model for
finding the search direction, associated with model
\eqref{form:m0_CC}, can be formulated as
\begin{equation}\label{form:mc} \tag{C}
\begin{array}{lll}
\minimize{t,\rho, p_k, B_k \succ 0} & \disp\sum_{i=0}^{k-1}t_i^2   \\
\subject & \mathbb{P}\left\{
\begin{array}{l}
  -t_i \leq (g_{i+1}- g_i)^T p_k \leq t_i, \ i = 0,\dots,k-2, \\
  -t_{k-1} \leq (\tilde{g}_k- g_{k-1})^T p_k \leq t_{k-1},\\
   B_k p_k = -\tilde{g}_k, \\
B_k = \tilde{V}_k +  \sum_{i=0}^{k-2} \rho_{i} (g_{i+1}-g_i)(g_{i+1}-g_i)^T \\
\phantom{B_k =}
 +\rho_{k-1} (\tilde{g}_k-g_{k-1})(\tilde{g}_k-g_{k-1})^T
\end{array}
\right\} \geq 1-\beta.
\end{array}
\end{equation}
In model \eqref{form:mc}, only $\tilde{g}_k$ is the random vector, while $g_i, i=0,\dots, k-2$, are constant values.


Propositions~\ref{Th_noise} and~\ref{Th_noiseCC} guarantee that the search directions obtained by solving model \eqref{form:m1} and model \eqref{form:mc} are descent directions under certain conditions, respectively.

\section{Reformulation for the low-rank quasi-Newton setting}\label{sec-sto}

Models \eqref{form:m1} and \eqref{form:mc} are hard to solve given the
non-convex nature of the equality-constraints created by the condition
$B_k p_k=-g_k$ and the associated condition on $B_k$, and the nature
of the chance constraints. In the low-rank setting, however, the
constraints related to the quasi-Newton matrix can be simplified significantly.

In our low-rank setting, the only unknown parameter is for the
memoryless BFGS matrix, where we may write
\[
 B_k = V_k +  \rho (g_k-g_{k-1})(g_k-g_{k-1})^T,
\]
and treat $\rho$ as a variable. For this case, we may allow further
simplification by circumventing the possible singularity of $V_k$ by
letting $\hat\rho_{k-1}$ be the value given by the secant condition
(\ref{eqn-secant}), and writing
\begin{equation}\label{eqn-Bk}
B_k  = \hat B_k + (\rho-\hat\rho_{k-1})  ( g_k - g_{k-1} ) ( g_k - g_{k-1})^T,
\end{equation}
for
\begin{equation}\label{eqn-barBk}
\hat B_k = V_k +  \hat\rho_{k-1} ( g_k - g_{k-1} ) ( g_k - g_{k-1} )^T.
\end{equation}
For the memoryless BFGS matrix, it holds that $V_k$ is positive
semidefinite with at most one zero eigenvalue in addition to $V_k
(g_k-g_{k-1})\ne 0$, so that $\hat B_k \succ 0$ as $\hat\rho_{k-1} >
0$. The point of introducing $\hat\rho_{k-1}$ and $\hat B_k$ is to give a
nonsingular and positive definite $\hat B_k$, which may be used as a
foundation for optimizing over $\rho$. We may therefore view the
optimization over $\rho$ as the potential for improving over utilizing
the secant condition.


For a $\hat\rho_{k-1}$ such that $\hat B_k \succ 0$ in (\ref{eqn-Bk}), the
Sherman-Morrison formula gives
\begin{equation}\label{eqn-Bkinv}
B_k\inv = \hat B_k^{-1} + \gamma
 \hat B_k^{-1} ( g_k - g_{k-1} ) ( g_k - g_{k-1} )^T\hat B_k^{-1}
\end{equation}
for
\begin{equation}\label{eqn-gamma}
\gamma=-\frac{(\rho-\hat\rho_{k-1})}{1+(\rho-\hat\rho_{k-1})\left(
    g_k - g_{k-1} \right)^T\hat B_k^{-1}\left( g_k - g_{k-1} \right)},
\end{equation}
so that an explicit expression for $p_k$ may be given as
\[
p_k = -B_k^{-1}g_k
 =-\hat B_k^{-1}g_k - \gamma ( g_k - g_{k-1})^T\hat B_k^{-1} g_k \hat B_k^{-1}
 ( g_k - g_{k-1} ).
\]
Note that there is a one-to-one correspondence between $\gamma$ and
$\rho$ as (\ref{eqn-gamma}) gives
\begin{equation}\label{eqn-rho}
\rho-\hat\rho_{k-1} =-\frac{\gamma}{1+\gamma\left(
    g_k - g_{k-1} \right)^T\hat B_k^{-1}\left( g_k - g_{k-1} \right)}.
\end{equation}
In addition, if $\hat B_k\succ0$, then (\ref{eqn-gamma}) and
(\ref{eqn-rho}) show that $B_k\succ 0$ if and only if the equivalent
conditions
\[
\gamma> -\frac1{(g_k - g_{k-1})^T\hat B_k^{-1}( g_k - g_{k-1} )}
\text{and}
\rho-\bar\rho> -\frac1{(g_k - g_{k-1})^T\hat B_k^{-1}( g_k - g_{k-1} )}
\]
hold. This is a consequence of these lower bounds defining an interval
around $\hat B_k$ and $\hat B_k\inv$ respectively, where $B_k$ and
$B_k\inv$ are well defined.

Summarizing, we may formulate the simplified problem as
\begin{equation}\label{form:P}\tag{DS}
\begin{array}{ll}
\minimize{\gamma, p_k}& \disp\sum_{i=0}^{k-1} ((g_{i+1}- g_i)^T p_k )^2 \\
\subject & p_k =-\hat B_k^{-1}g_k
 -\gamma (g_k - g_{k-1} )^T\hat B_k^{-1}
 g_k\hat B_k^{-1} ( g_k - g_{k-1} ), \\[1mm]
&\gamma >
  -\frac{1}{\left( g_k - g_{k-1} \right)^T \hat B_k^{-1}\left( g_k - g_{k-1}
    \right)},
\end{array}
\end{equation}
which is a convex constrained quadratic program if a tolerance is
introduced for the strict lower bound on $\gamma$. For this problem,
we may eliminate $p_k$ to get one variable only, $\gamma$. Note the
one-to-one correspondence given by (\ref{eqn-rho}) which allows us to
recover $\rho$ from $\gamma$.

For the chance-constrained model \eqref{form:mc}, analogous to
\eqref{eqn-Bklimmemory_Random}, we denote
\begin{equation}\label{eqn-Bk_Random}
B_k  = \tilde B_k + (\rho-\hat\rho_{k-1})  ( \tilde g_k - g_{k-1} ) ( \tilde
g_k - g_{k-1})^T,
\end{equation}
for
\begin{equation}\label{eqn-barBk_Random}
\tilde B_k = V_k +  \hat\rho_{k-1} ( \tilde g_k - g_{k-1} ) ( \tilde
g_k - g_{k-1} )^T,
\end{equation}
which is random due to the randomness of $\tilde g_k$.
$V_k$ is deterministic as defined in \eqref{eqn-Vkmemoryless}.
Then, analogous
simplification and reformulation of (\ref{form:mc}) gives
\begin{equation}\label{form:ppb-sc}\tag{CS}
\begin{array}{llll}
        \minimize{t,\gamma,p_k} & \disp\sum_{i=0}^{k-1} t_i^2  \\
         \subject & \disp\mathbb{P} \left\{
        \begin{array}{l}
        -t_i \le (g_{i+1}- g_i)^T p_k \le t_i, \ i = 0,\dots,k-2, \\ [1mm]
        -t_{k-1} \le (\tilde{g}_k- g_{k-1})^T p_k \le t_{k-1}, \\
p_k =-\tilde B_k^{-1}\tilde{g}_k \\
\phantom{p_k =}
 -\gamma ( \tilde{g}_k - g_{k-1} )^T\tilde B_k^{-1}
 \tilde{g}_k\tilde B_k^{-1} ( \tilde{g}_k - g_{k-1} ), \\[1mm]
\disp \gamma >
  -\frac{1}{\left( \tilde{g}_k - g_{k-1} \right)^T \tilde B_k^{-1}\left( \tilde{g}_k - g_{k-1}
    \right)}
        \end{array}
        \right\} \ge 1-\beta, \\
\end{array}
\end{equation}




Our proposed model \eqref{form:ppb-sc} can be read as follows: the
obtained solution $\gamma$, which will be transformed to $\rho$ by
(\ref{eqn-rho}), will have a probability of at least $1-\beta$ to
obtain a direction $p_k$ from $\tilde{g}_k$ that will be a descent
direction. This means our approach obtains a value of $\gamma$ robust
enough to point us in a descent direction $p_k$ with a probability
bounded given the probabilistic nature of the gradient.

Chance constrained models are often non-convex in general and hard to
solve \cite{SDR14}. However, different equivalent formulations can be
applied to obtain an analytical solution or approximate the chance
constraints.
In general, it is often difficult to get an analytical solution, since it always requires strict assumptions on the probability distribution of random variables and the structure of chance constraints, which making the situation specific and not general enough.
In constrast to the analytical solution, sample average approximation and scenario approximation are two general approaches without much assumptions on the random variables, which will be applied to solve the chance constrained model in the following sections.

\subsection{Deterministic equivalent formulation}

Model \eqref{form:ppb-sc} can not be solved directly in its current state, as it is not a deterministic problem.
Therefore, a sample average approximation (SAA) approach is proposed, given its flexibility to work under any type of stochastic variables. The first step is to formulate it as a deterministic problem that approximates the solution of \eqref{form:ppb-sc} and that can be solved by some solvers. Let $\Omega$ be the set of sample, $|\Omega|=S$, $\tilde{g}_{k\omega}, \omega \in \Omega$ be the i.i.d. noisy gradient samples, $\delta>0$ a sufficiently small real number and $0 \le K < k$ be an integer number that represent the time window to be considered in the model. All samples have a probability of $1/S$. Then, a deterministic equivalent formulation of \eqref{form:ppb-sc} using sample average approximation (SAA) is as follows:

\begin{equation}\label{form:csa}\tag{CSA}
\begin{array}{lll}
        \minimize{}  & \sum_{i=\max(0,k-K)}^{k-1} t_i^2  \\
         \subject     &
        (g_{i+1}- g_i)^T p_{k\omega} \ge -t_i - Mz_\omega,  \ i \in
        I_K, \ \omega \in \Omega,\\
       & (g_{i+1}- g_i)^T p_{k\omega} \le t_i + Mz_\omega, \ i \in
       I_K, \ \omega \in \Omega,\\
       & -t_{k-1}- Mz_\omega \le (\tilde{g}_{k\omega}- g_{k-1})^T
       p_{k\omega} \le t_{k-1} + Mz_\omega, \ \omega \in \Omega, \\
&p_{k\omega} =-\tilde B_{k\omega}^{-1}\tilde{g}_{k\omega} \\
&\phantom{p_{k\omega} =}
 -\gamma ( \tilde{g}_{k\omega} - g_{k-1} )^T\tilde B_{k\omega}^{-1}
 \tilde{g}_{k\omega}\tilde B_{k\omega}^{-1} ( \tilde{g}_{k\omega} -
 g_{k-1} ), \ \omega \in \Omega,\\[1mm]
&\disp \gamma + Mz_\omega \ge
  -\frac{1}{\left( \tilde{g}_{k\omega} - g_{k-1} \right)^T \tilde B_{k\omega}^{-1}\left( \tilde{g}_{k\omega} - g_{k-1}
    \right)} + \delta, \ \omega \in \Omega,\\
    & \sum_{\omega \in \Omega} z_\omega \leq \lfloor S\beta \rfloor,\\
            & z_\omega \in \{0,1\}, \ \omega \in \Omega,
\end{array}
\end{equation}
where $I_K = \{\max(0,k-K),\dots,k-2\}$ and $K \le k$ indicates the
number of gradients to be considered in the model. Once we obtain the
optimal value of $\gamma$, then $\rho_k$ is approximated by

\begin{equation}\label{eqn-rhocsa}
\rho_{k-1} =\frac{\gamma}{1+\gamma\left(\bar{g}_k - g_{k-1} \right)^T \bar{B}_k^{-1} \left(\bar{g}_k - g_{k-1} \right)},
\end{equation}
where $\bar{g}_k = \frac{1}{S}\sum_{\omega \in \Omega}g_{k\omega}$ and
$\bar{B}_k$ is obtained by replacing $\tilde{g}_k$ in
\eqref{eqn-barBk_Random} with $\bar{g}_k$. Finally, $p_k$ is obtained
using the equation
\begin{equation}\label{eqn-pkcsa}
p_k = -B_k^{-1}\bar{g}_k
 =-\bar B_k^{-1}\bar{g}_k - \gamma ( \bar{g}_k - g_{k-1})^T\bar B_k^{-1} \bar{g}_k \bar B_k^{-1}
 ( \bar{g}_k - g_{k-1} ),
\end{equation}
i.e., using the average of the gradients at iteration $k$.

It can be noticed that if $\beta > 0$ then (\ref{form:csa}) is a mixed
integer program whose complexity will be tied to the number of
dimensions and samples used to solve the problem.  Since at every
iteration new gradients are added, the dimensionality of the problem
will grow at each step regardless. And to guarantee the quality of
solution by SAA, the sample size should not be too small if the
dimension is large.

Therefore the complexity of this approach grows at each step, leading
to increasing solving times which will be an issue on long runs with
low convergence speed. However, by using the value of parameter $K$
with values greater than 0, we can use a limited memory or memoryless,
implementation to avoid this.

If $\beta  = 0$, then all binary variables must be set to zero and can
be eliminated from the problem, creating a continuous linear program which is much simpler to solve. This is commonly referred to the scenario approach, where all possible sampled scenarios of the random variables are being considered. This also implies that the solution will be closely tied to the most conservative of the sampled scenarios. There are other approaches to simplify and approximate this model present in the literature \cite{Ahm18}.

\section{Computational results for the quasi-Newton methods}


Two sets of results are presented: First, we consider a set of randomly generated problems intended to illustrate the properties and methodologies proposed in this paper. The second set of problems are real-life instances from the CUTEst test set~\cite{GOT15}, to test the applicability of these methods in a more realistic environment.

A comparison of the results is provided using different models and/or approximation formulations, and we discuss the practical implications obtained with each method. All models are implemented in Python 3.7.10, using Gurobi 9.1 as a solver for the resulting optimization problems and all computation were done on an Intel(R) i7 @ 2.7 GHz and 16 GB of memory over macOS 10.

The methods chosen in our experiments are as follows:

\begin{itemize}
	\item Steepest Descent (SD):
	\[B_k = I.\]
	\item Conjugate Gradients (CG): The symmetric CG method as presented in \eqref{eqn-BksymCG},
	\[B_k = \left( I - \frac{1}{g_{k-1}^T p_{k-1}\drop}g_k p_{k-1}^T\right) \left( I - \frac{1}{g_{k-1}^T p_{k-1}\drop} p_{k-1}\drop g_k^T\right).\]
	\item BFGS, as presented in \eqref{eqn-BFGSfull} and where $\rho_{k-1}$ is given by the secant condition in \eqref{eqn-secant},
	\[B_k = B_{k-1} + \frac1{g_{k-1}^T p_{k-1}} g_{k-1}g_{k-1}^T + \frac1{\alpha_{k-1}(g_k-g_{k-1})^Tp_{k-1}}(g_{k}-g_{k-1})(g_{k}-g_{k-1})^T.\]
	\item Memoryless-BFGS (ml-BFGS), as presented in \eqref{eqn-Bkmemoryless} and where $\rho_{k-1}$ is given by the secant condition \eqref{eqn-secant},
	\[B_k = I - \frac1{p_{k-1}^T p_{k-1}\drop} p_{k-1} p_{k-1}^T + \rho_{k-1} ( g_k - g_{k-1} ) ( g_k - g_{k-1})^T.\]
      \item Chance-Constrained Quasi Newton (CCQN $\beta$). The search
        direction is obtained by first solving \eqref{form:csa} for
        $K=0$ and $\beta$, then obtaining $\rho_{k-1}$ from \eqref{eqn-rhocsa}
        and finally $p_k$ from \eqref{eqn-pkcsa}.
	\item Limited Memory Chance-Constrained Quasi Newton
          (lm-CCQN $\beta$). Same as CCQN, but $0<K<k$.
\end{itemize}

In order to standardize our results across different experiments, \emph{performance profiles} are used in two different analyses: First, for the number of steps required by the method to break a certain gradient norm value which will define as the tolerance level (denoted as $tol$), then by taking the method with the lowest possible value of steps as the comparison point, we show how much larger the values of the other methods are compared to this minimum. Next, a performance profile is created to detect the minimum value of the gradient norm. Finally, the total amount of times the different experiments using a set method were able to reach a set thresholds is summed and presented: Starting from 1 (being the minimum) up to 20 (i.e. 20 times the value of the minimum). This methodology is fully expanded in \cite{EF21}.

Since noise will severely distort the gradient norms once it reaches a certain point in the run, a set of performance profiles are created for tolerance values close to the noise variance, i.e. if the variance $\sigma^2=10^{-2}$ then these performance profiles will be set to tolerances such as $10^{-1}, 10^{-2}$ and $10^{-3}$. Higher does not show significant differences to the deterministic case, and lower can cause the method to not reach any threshold.

For every experiment, we applied the following algorithm: At iteration $k$, $tol$ denotes our tolerance threshold of precision for the norm of the gradient of the solution obtained, $\bar{g}_k$ is the average value of the gradient $\tilde{g}_k$, $K \le k$ denotes the number of gradients to be used in the calculations and $MaxK$ is the maximum number of steps. When calculating the step length $\alpha$, an exact line search approach is used with the value of the gradient without noise (the deterministic value of $g_k$), as the objective is to isolate the effects of each method in finding a descent direction. The algorithm goes as follows:

\begin{algorithm}
\caption{\label{alg-general} General solving algorithm.}
\begin{algorithmic}[0]
  \State $k \gets 0$;
  \State $x_k \gets $ \mbox{initial point};
  \State $\tilde{g}_{k\omega} \gets Hx_k+c + \epsilon_\omega, \quad \omega \in
  \Omega$;
  \State $\bar{g}_k \gets \frac{1}{S}\sum_{\omega \in \Omega}\tilde{g}_{k\omega}$;
         \While{$\norm{g_k}_2 > tol \ \AND \ k \le MaxK$}
         \If{\texttt{method} = \texttt{CCQN}}
  	\State      $\gamma \gets $ solution to \eqref{form:csa} using $K$ gradients;
    \State      $\rho \gets $ solution to \eqref{eqn-rhocsa};
	\EndIf
  \State	  $B_k \gets$ from \texttt{method};
  \State      $p_k \gets $ solution to $B_k p_k = - \bar{g}_k$;
  \State      $\alpha_k \gets - \frac {g_k^T p_k}{p_k^T H p_k}$;
  \State      $x_{k+1} \gets x_k + \alpha_k p_k$;
  \State      $k \gets k + 1$;
  \State      $\tilde{g}_{k\omega} \gets Hx_k+c + \epsilon_\omega, \quad \omega \in
  \Omega$;
  \State $\bar{g}_k \gets \frac{1}{S}\sum_{\omega \in \Omega}\tilde{g}_{k\omega}$;
  \EndWhile
\end{algorithmic}
\end{algorithm}

All experiments are repeated 30 times using different random number generator seeds and using 20 samples of noisy gradients at each step, therefore the performance profiles also separate each method and experiment by seed. We used a maximum amount of steps $MaxK = 500$ and $K=10$ for the lm-CCQN method.

\subsection{Results for randomly created problems}\label{sec-smallprob}

The first experiment is a set of randomly generated unconstrained quadratic problems. For each problem, the Hessian matrix $H$ is defined as $H = Q^T Q + \epsilon \diag(U_{1,n})$, where $Q = a J_{1,n} + (b-a) U_{n,n}$, $a,b \in \mathbb{R}$,  $J_{n,m}$ is the unit $n\times m$ matrix, $U_{n,m}$ is a $n \times m$ matrix, each component of $U_{n,m}$ is randomly generated following a Uniform(0,1) distribution and $\epsilon>0$ is a sufficiently small number. The vector $c$ is randomly generated as $c = U_{n,1}$. In our experiments, we defined $a=-1, b=1, n = 100$ and $\epsilon = 0.3$.

\begin{figure}[H]
                \centering
                \subfloat[$\sigma^2=10^{-6}$]{{\includegraphics[width=6.5cm]{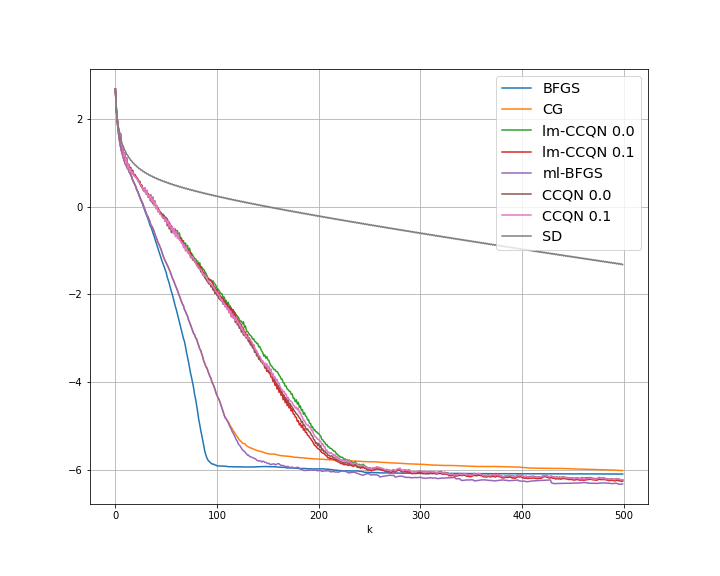}}}%
                \qquad
                \subfloat[$\sigma^2=10^{-2}$]{{\includegraphics[width=6.5cm]{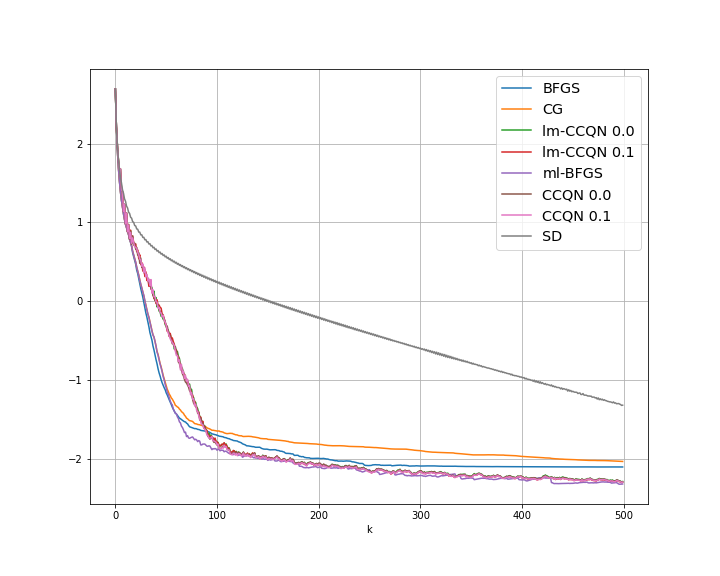}}}%
                 \caption{Average log norm of the gradient at step $k$ for each tested method with different noise variances.}
                \label{graph:models-lognorm}
\end{figure}

Figure \ref{graph:models-lognorm} shows the behaviour of each method in the two selected noise levels. The traditional approaches, such as CG or BFGS, appear to converge faster but the average log norm of the gradient can not surpass the tolerance threshold $tol $ smaller than the noise variance $\sigma^2$, while for CCQN and lm-CCQN (regardless of the value $\beta$) the average log norm of the gradient can surpass this barrier, albeit slower compared to the results presented by ml-BFGS.

In this experiment, the chosen value of $MaxK$ was not large enough for SD to show convergence as seen in figure \ref{graph:models-lognorm}, however we ran the same experiment for this method using a larger value, showing that average log norm of the gradient found by SD can suprass the $tol \le \sigma^2$ barrier, similar to CCQN, lm-CCQN and ml-BFGS.

\captionsetup[subfigure]{labelformat=empty}
\begin{figure}[H]
                \centering
                \subfloat[$tol = 10^{-5}, \sigma^2 = 10^{-6}$]{{\includegraphics[width=6.5cm]{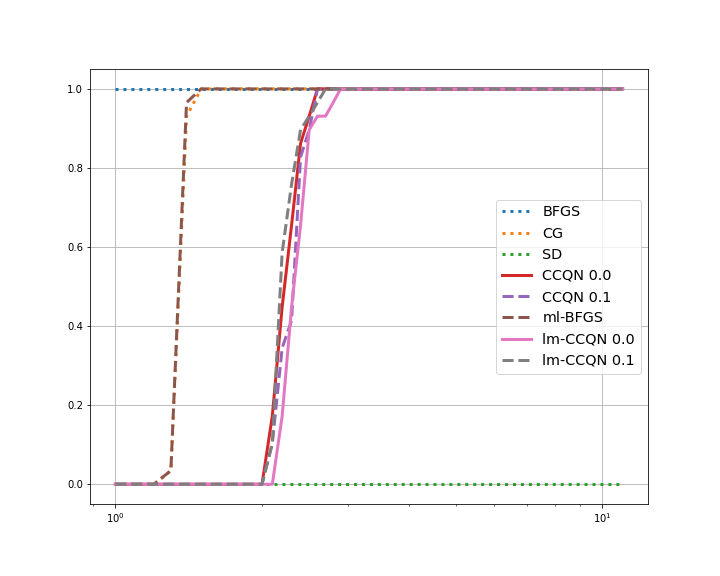}}}%
                \qquad
                \subfloat[$tol = 10^{-1}, \sigma^2 = 10^{-2}$]{{\includegraphics[width=6.5cm]{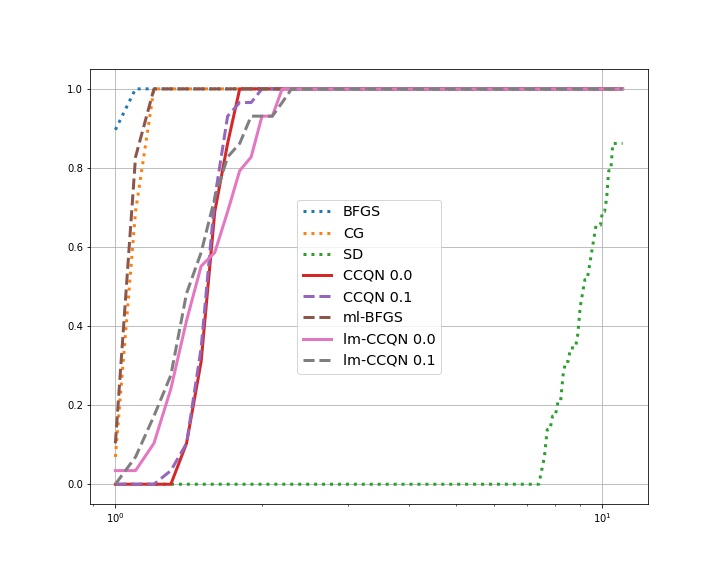}}}%
                \\
                \subfloat[$tol = 10^{-6}, \sigma^2 = 10^{-6}$]{{\includegraphics[width=6.5cm]{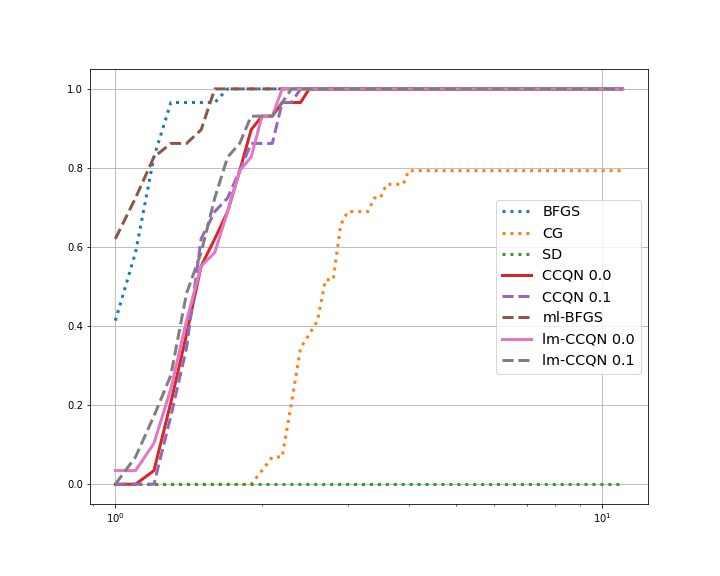}}}%
                \qquad
                \subfloat[$tol = 10^{-2}, \sigma^2 = 10^{-2}$]{{\includegraphics[width=6.5cm]{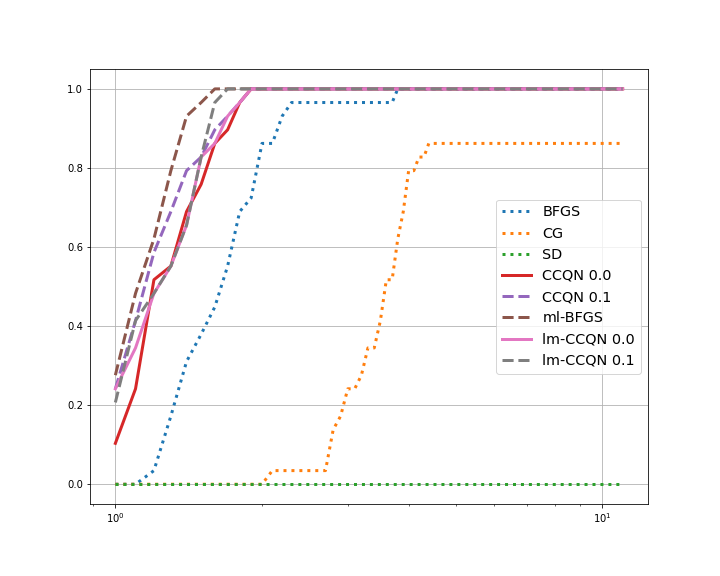}}}%
                 \caption{Performance profiles for different tolerances and noise variance levels.}
                \label{graph:models-pp}
\end{figure}

Figure \ref{graph:models-pp} shows the performance profiles of each method for two noise variance levels under two tolerance thresholds. In the traditional approaches, we can observe that ml-BFGS performs better overall, while CG and BFGS perform well under low noise variance but poorly under larger values. On the other hand, the performance of CCQN and lm-CCQN methods do not show significant differences for different tolerance and noise variance levels. Furthermore, the difference between using $\beta > 0$ or 0 is not significant, which means we implement a convex approximation using the scenario approach ($\beta=0$), avoiding the usage of a mixed integer linear program which  becomes difficult to solve for larger problems.


\captionsetup[subfigure]{labelformat=parens}
\begin{figure}[H]
                \centering
                \subfloat[$\sigma^2 = 10^{-6}$]{{\includegraphics[width=6.5cm]{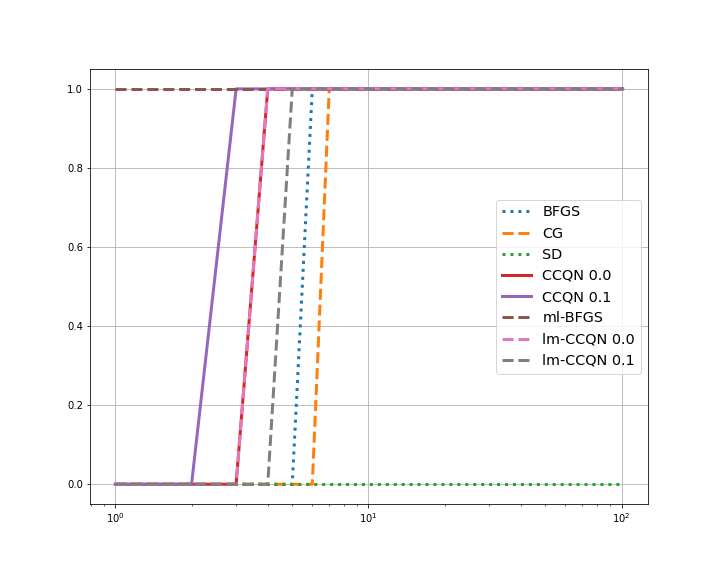}}}%
                \qquad
                \subfloat[$\sigma^2 = 10^{-2}$]{{\includegraphics[width=6.5cm]{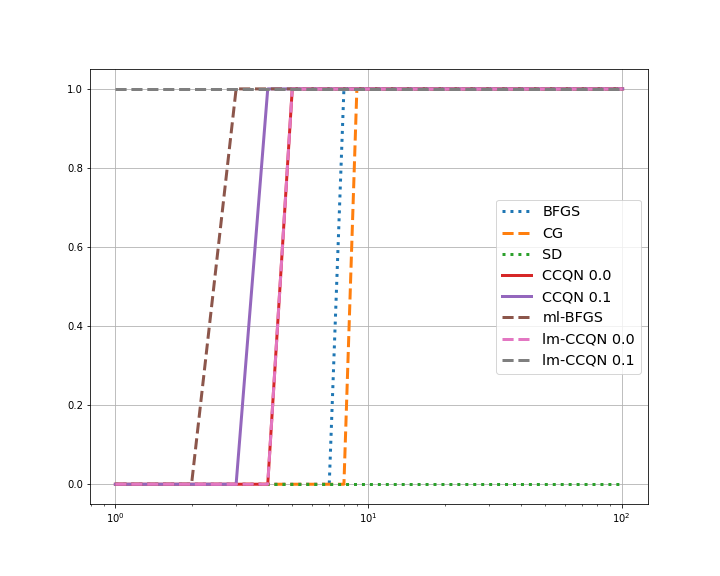}}}%
                \caption{Performance profile of the minimum gradient norm for different noise variance levels.}
                \label{graph:models-smallmin}
\end{figure}

Figure \ref{graph:models-smallmin} shows the performance profiles of the minimum gradient norm found in the set of problems with different seeds for two noise variance levels. When $\sigma^2=10^{-6}$, we can observe that ml-BFGS is able to obtain the minimum value consistently, i.e. for every problem and every seed, which is followed by CCQN and lm-CCQN. When $\sigma^2=10^{-2}$, lm-CCQN performs better than ml-BFGS, and followed by CCQN. The classical methods BFGS, CG and SD perform worse than the previously discussed methods independent of the noise variance, and specifically the minimum gradient norms found by SD are larger than any of minimum norms found by all the other methods.

\subsection{Results for CUTEst problems}\label{sec-qts}

In our experiments, we will compare the performance of the same
approaches presented in the last section applied to different problems
from the CUTEst test set \cite{GOT15}, specifically quadratic,
unconstrained and number of variables chosen by the user (\texttt{QUV}
using CUTEst classification system). However, only 6 problems families
fall in this category, therefore we implemented a second batch of
problems by adding those unconstrained sum of squares problems
(\texttt{SUV} using CUTEst classification system) which had a positive
definite Hessian at the starting point and left it constant throughout
the solving scheme. This brought the total amount of problems to
22. In our numerical experiments, results for noise variance $10^{-2}$
and $10^{-6}$ did not show significant differences. Therefore, in this section we only focus on results with a noise variance of $10^{-2}$.

\begin{figure}[H]
                \centering
                \subfloat[$tol = 10^0$]{{\includegraphics[width=6.5cm]{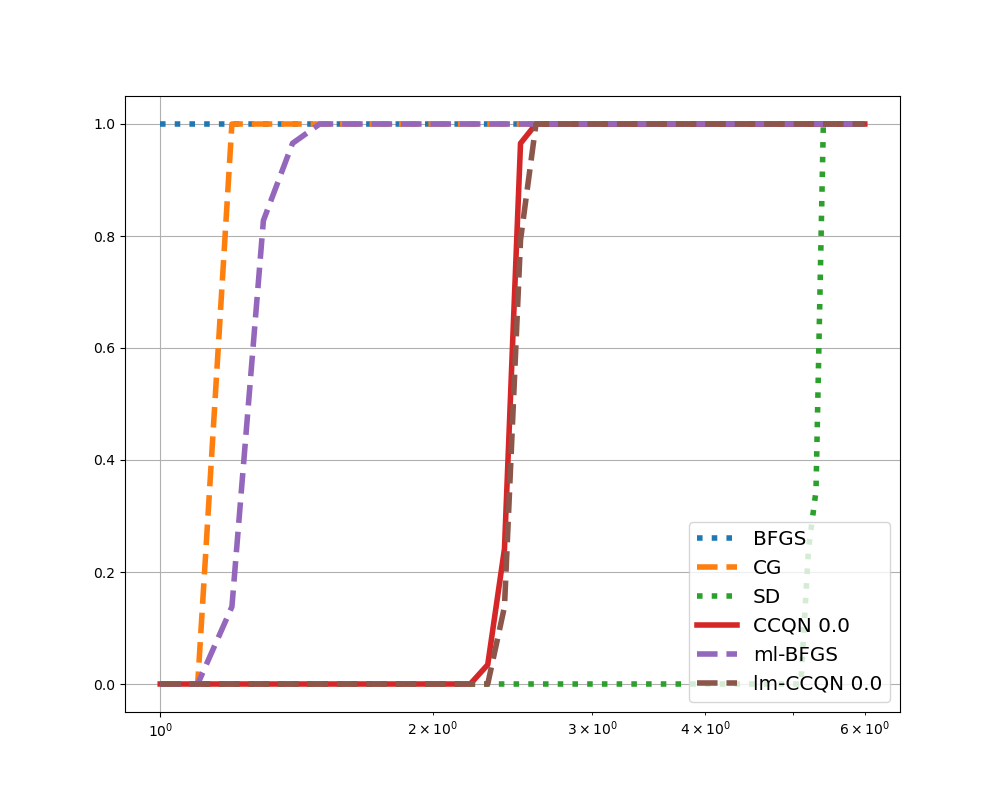}}}%
                \qquad
                \subfloat[$tol = 10^{-1}$]{{\includegraphics[width=6.5cm]{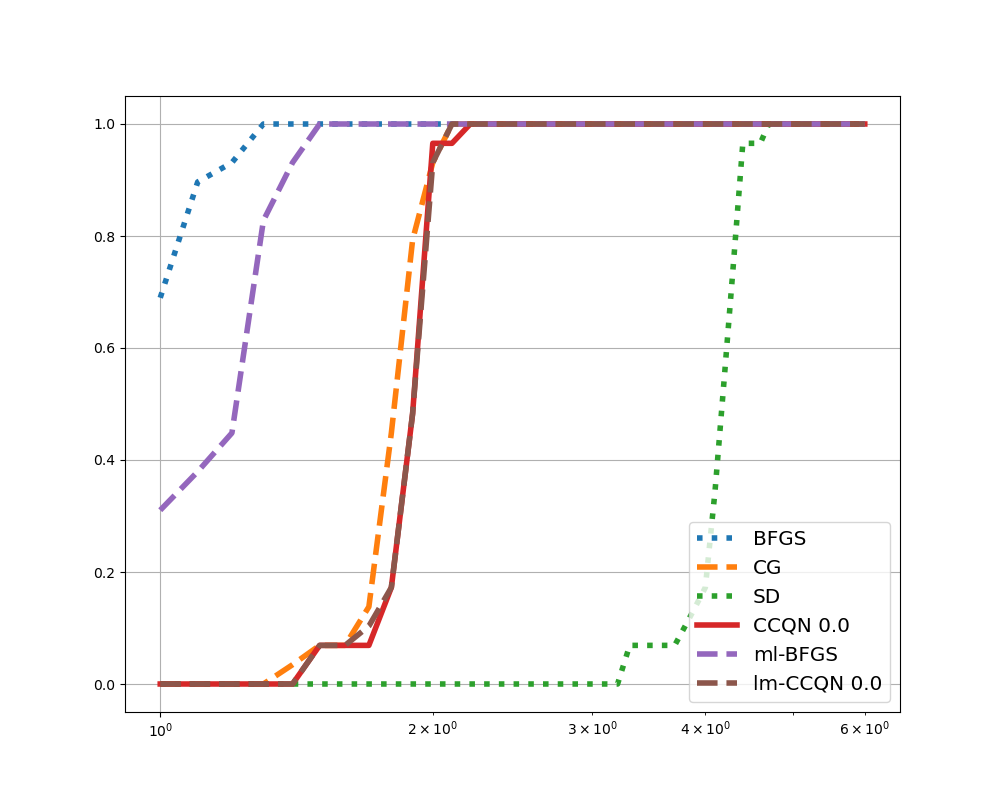}}}%
                \\
                \subfloat[$tol = 10^{-2}$]{{\includegraphics[width=6.5cm]{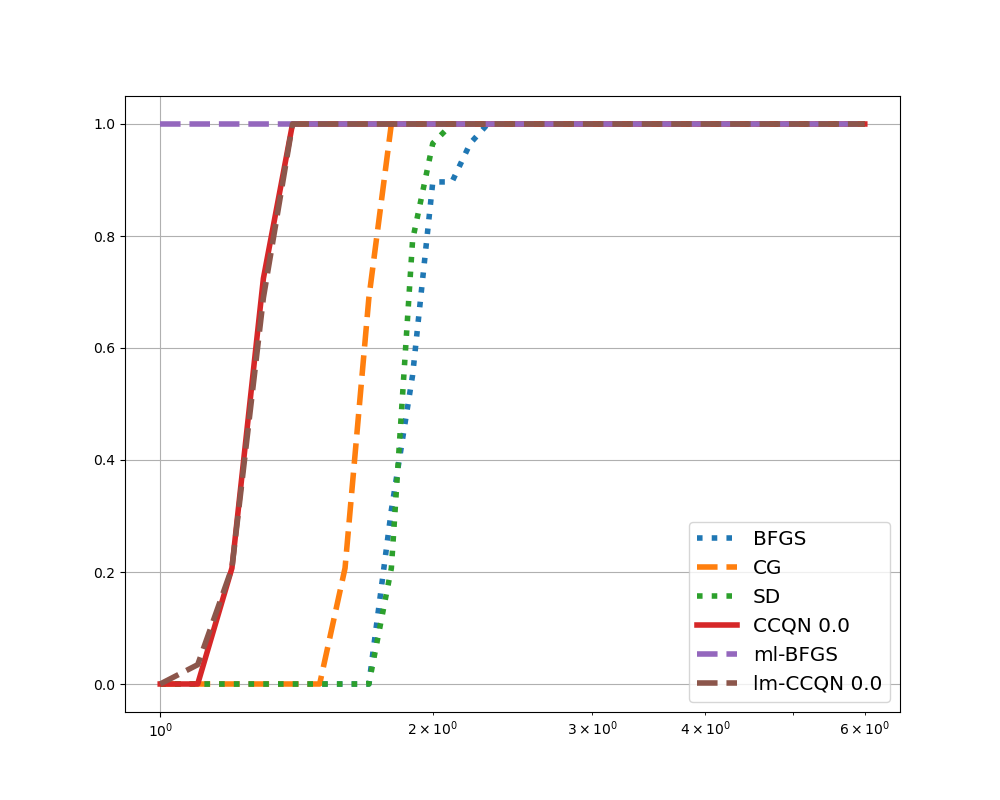}}}%
                \qquad
                \subfloat[$tol = 10^{-3}$]{{\includegraphics[width=6.5cm]{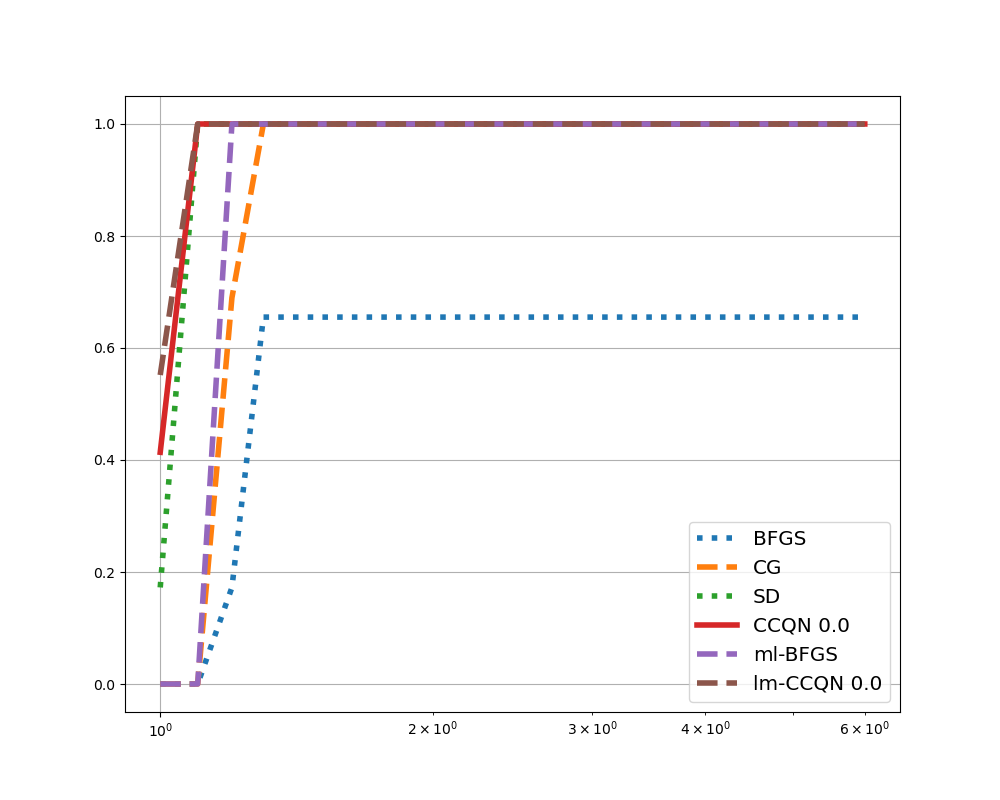}}}
                \caption{Performance profiles for different tolerance levels of the CUTEst problems.}
                \label{graph:cute-pp1}
\end{figure}

Figure \ref{graph:cute-pp1} shows the performance profiles of different methods solving CUTEst problems under different tolerance levels.
We can observe that the performance of BFGS and CG can be good and SD performs worst among all the tested methods, when the tolerance level is larger than the noise variance.
However, as shown in Figure \ref{graph:cute-pp1} (d),
SD can perform better than ml-BFGS, BFGS and CG.
ml-BFGS can show its efficiency and effectiveness in most cases.
CCQN and lm-CCQN present their robustness under different tolerance levels, and perform better when the tolerance level is close or smaller than the noise variance.
When $tol = 10^{-3}$, CCQN and lm-CCQN perform best, with the latter being slightly ahead.

When the tolerance level becomes much smaller, the performance profile of each method will depend on the problem, except CCQN and lm-CCQN. The ml-BFGS method will not always be the best, while the performances of CCQN and lm-CCQN are among the top three methods, which shows their robustness if the noise variance is larger the tolerance level. This behaviour can be observed in Appendix \ref{sec-appendix-a} for each individual problem.


\begin{figure}[H]
                \centering
                \includegraphics[width=8cm]{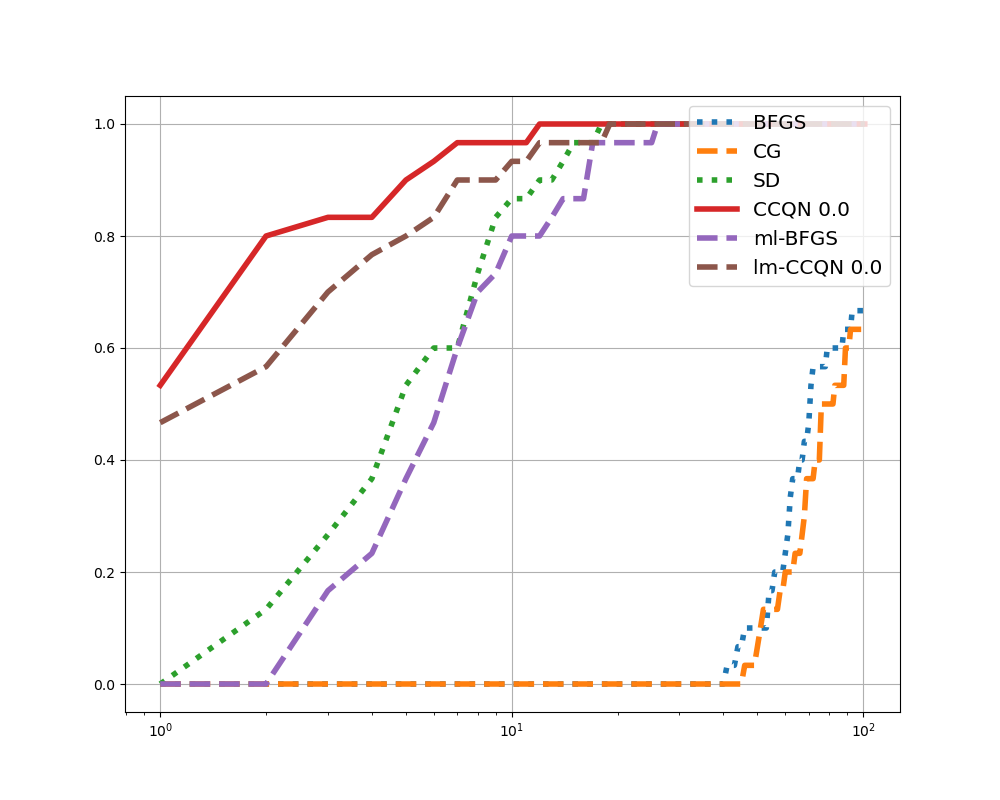}
                \caption{Performance profile of the minimum gradient norm found of the CUTEst problems.}
                \label{graph:models-cutemin}
\end{figure}

Figure \ref{graph:models-cutemin} the performance profiles of the minimum gradient norm found in the CUTEst set of problems with different seeds. CCQN is able to obtain the minimum value consistently, followed by the
lm-CCQN, SD and ml-BFGS (in that order), while BFGS and CG perform worst than all the other studied methods. The chance-constrained methods were able to obtain the minimum value faster and more frequently, even if compared it to ml-BFGS which had a good performance across our previously shown profiles.

\section{Conclusion}

In this paper, we have studied low-rank quasi-Newton methods for
minimizing a strictly convex quadratic function in a noisy framework.
We have considered a memoryless BFGS method and compared to a BFGS
method, the method of conjugate graddients and steepest descent.
In order to potentially improve the performance of the
low-rank quasi-Newton method, a chance constrained stochastic
optimization model has also been formulated.
The secant condition is here replaced by solving a one-dimensional convex
quadratic programming problem.  
The proposed chance constrained model, which can be solved effectively by sample average approximation method or scenario approach, has been proven to provide a descent search direction with a high probability in the random noisy framework, while the deterministic model may fail to provide a descent direction,
if the noise level is large.

In the numerical experiments, we compare classical methods and the
proposed chance constrained model in a noisy setting.  Results of
ml-BFGS and CCQN show promise when solving problems with uncertainty
in the gradient, however the latter is more consistent and its
performance appear to be independent of the problem, while the former
does not.  The performance of chance-constrained model (and its
different iterations) appears to be in the top three in terms of
convergence speed under different tolerance thresholds.
Furthermore, while studying the behaviour of all the models, the minimal value of gradient norm was consistently found by the approach based on chance constrained model.
Therefore, we believe that the usage of more advanced solving algorithms than the one presented (i.e. stochastic inexact linesearch) could further improve the results presented in this paper.

Finally, our intention is to investigate the behavior and the interplay between quality and robustness of the low-rank quasi-Newton method, especially in the case of large noise and multiple copies of gradients.
Both the theoretical and numerical results show that we can gain the robustness and accuracy of the computed solution with the chance constrained model, although the computational cost can be high.
This shows the potential to be further considered and explored in convex optimization problems.

\section*{Acknowledgement}
This preprint has not undergone peer review or any post-submission improvements or corrections. 
The Version of Record of this article is published in Computational Optimization and Applications, and is available online at https://doi.org/10.1007/ s10589-025-00661-4

\bibliography{refs,references,references2}
\bibliographystyle{myalpha}

\newpage
\appendix
\section{Average log norms gradients of the CUTEst problems}\label{sec-appendix-a}
In this section, the average log norm of each CUTEst problem is presented. The objective is to make evident the difference between problems as stated in the experiment section \ref{sec-qts}. All of these results are presented for $\sigma^2=10^{-2}$.

\captionsetup[subfigure]{labelformat=empty}

\begin{figure}[H]
                \centering
                \subfloat[\texttt{ARGLINA}]{{\includegraphics[width=6.5cm]{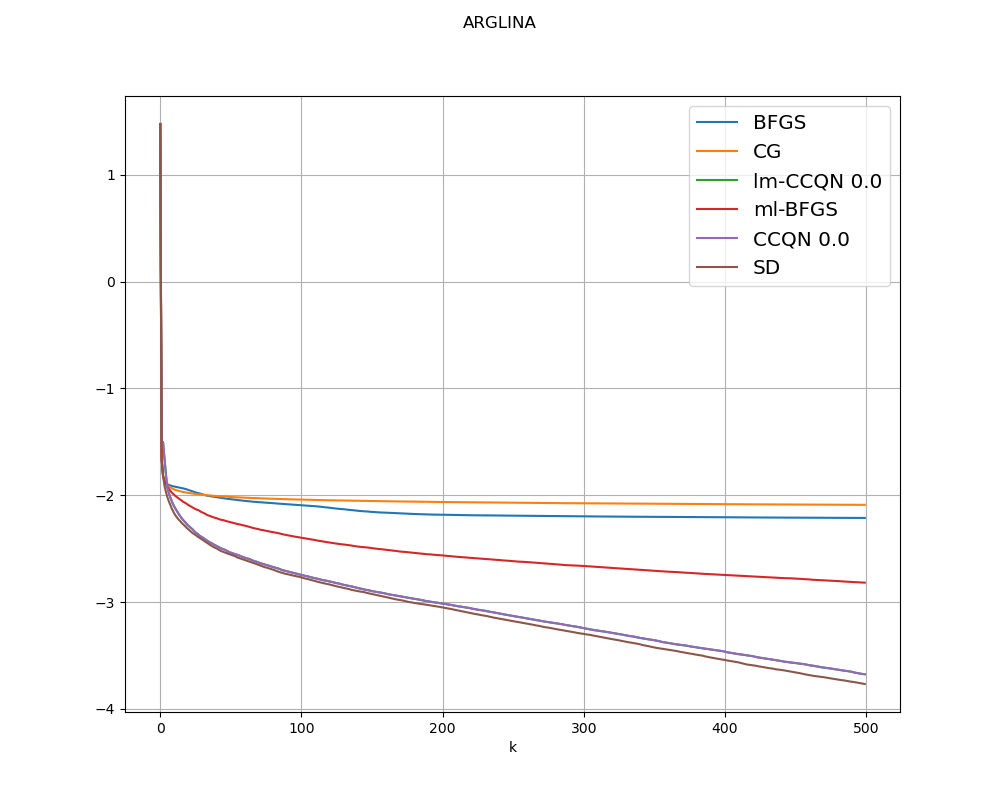}}}%
                \qquad
                \subfloat[\texttt{BDQRTIC}]{{\includegraphics[width=6.5cm]{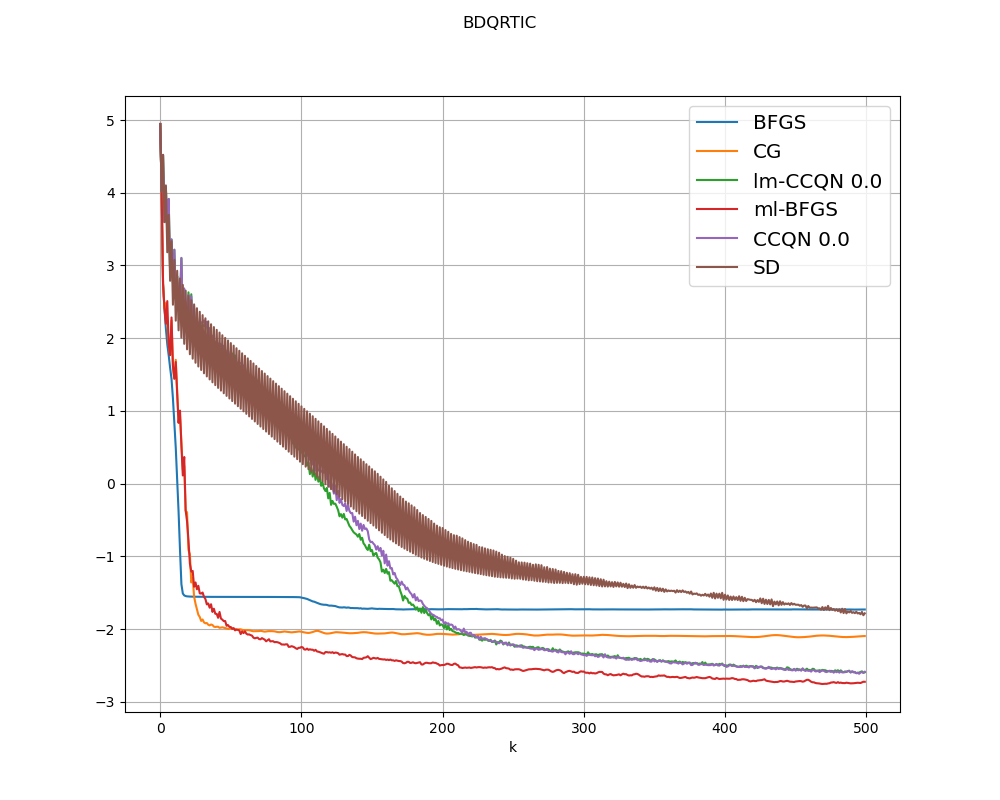}}}%
                \caption{Average log norm of the gradient at step $k$ for each tested method.}
\end{figure}

\begin{figure}[H]
                \centering
                \subfloat[\texttt{CHAINWOO}]{{\includegraphics[width=6.5cm]{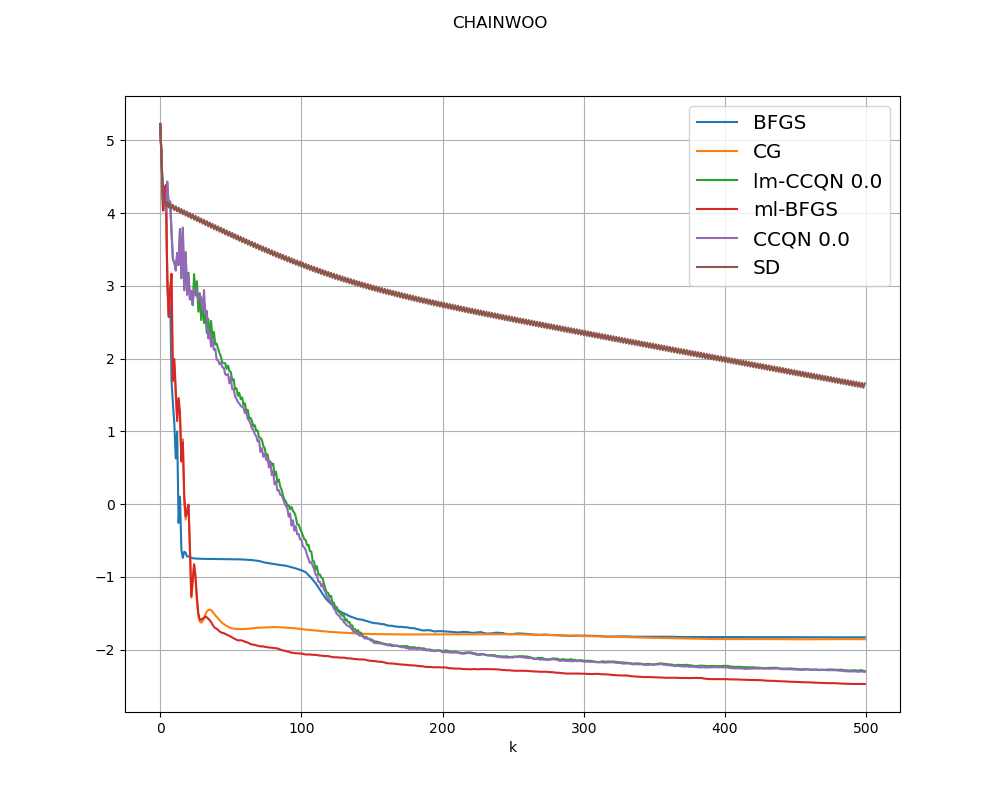}}}%
                \qquad
                \subfloat[\texttt{EXTROSNB}]{{\includegraphics[width=6.5cm]{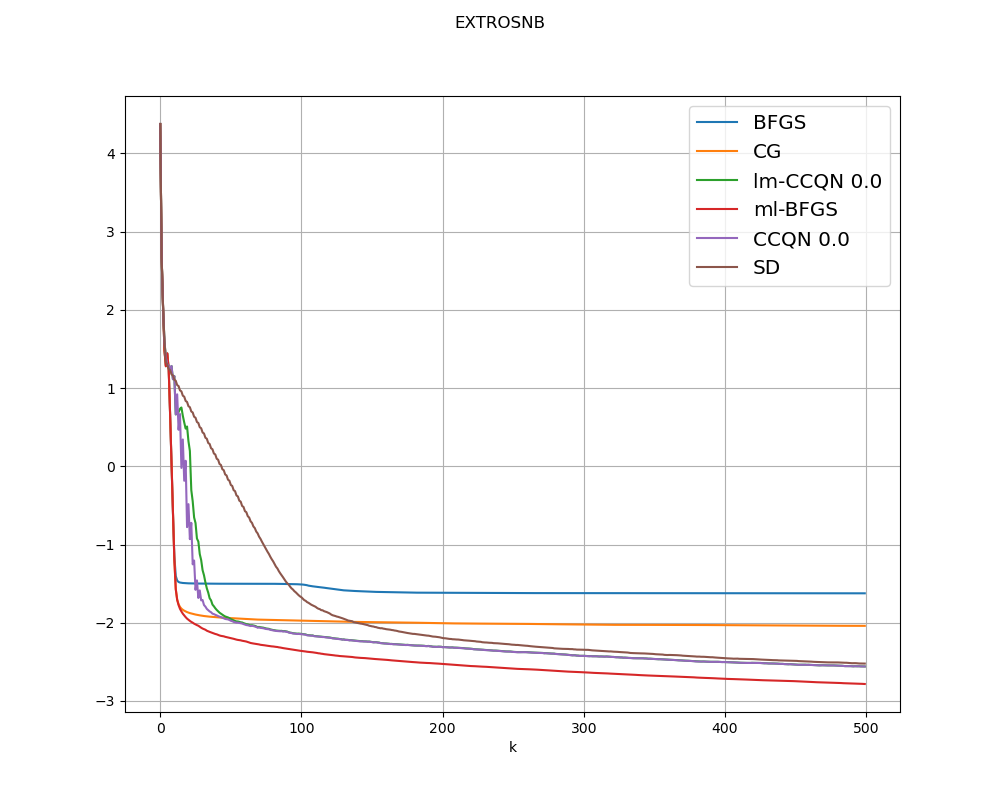}}}%
                \caption{Average log norm of the gradient at step $k$ for each tested method.}
\end{figure}

\begin{figure}[H]
                \centering
                \subfloat[\texttt{LIARWHD}]{{\includegraphics[width=6.5cm]{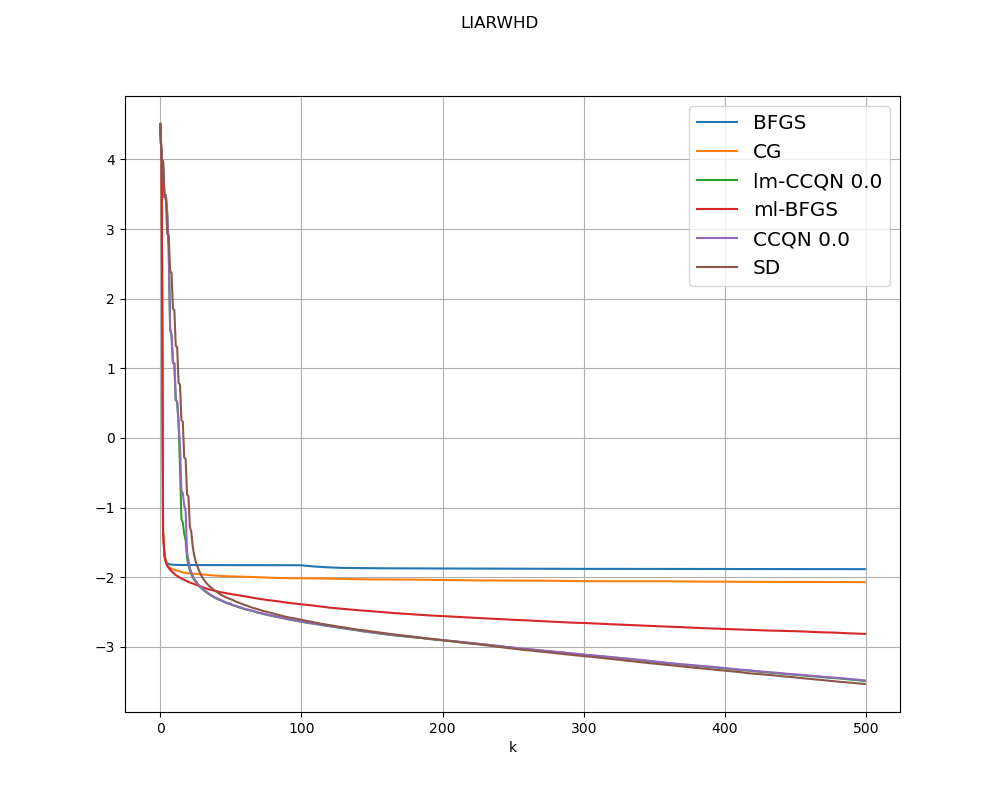}}}%
                \qquad
                \subfloat[\texttt{PENALTY1}]{{\includegraphics[width=6.5cm]{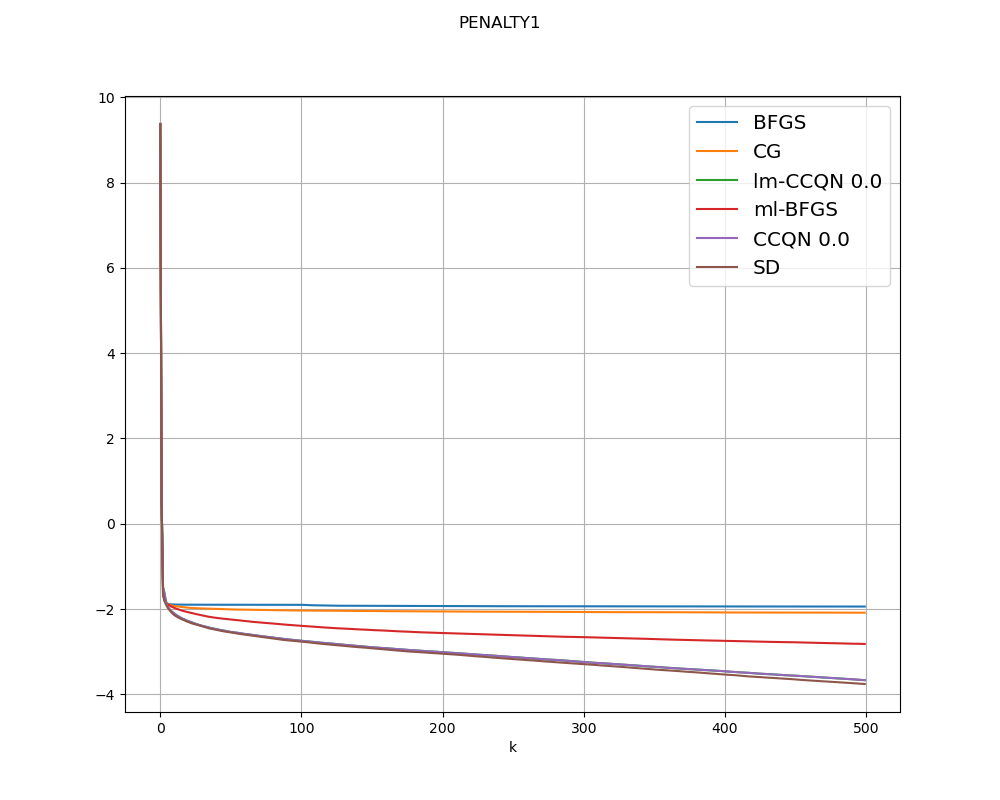}}}%
                \caption{Average log norm of the gradient at step $k$ for each tested method.}
\end{figure}

\begin{figure}[H]
                \centering
                \subfloat[\texttt{PENALTY2}]{{\includegraphics[width=6.5cm]{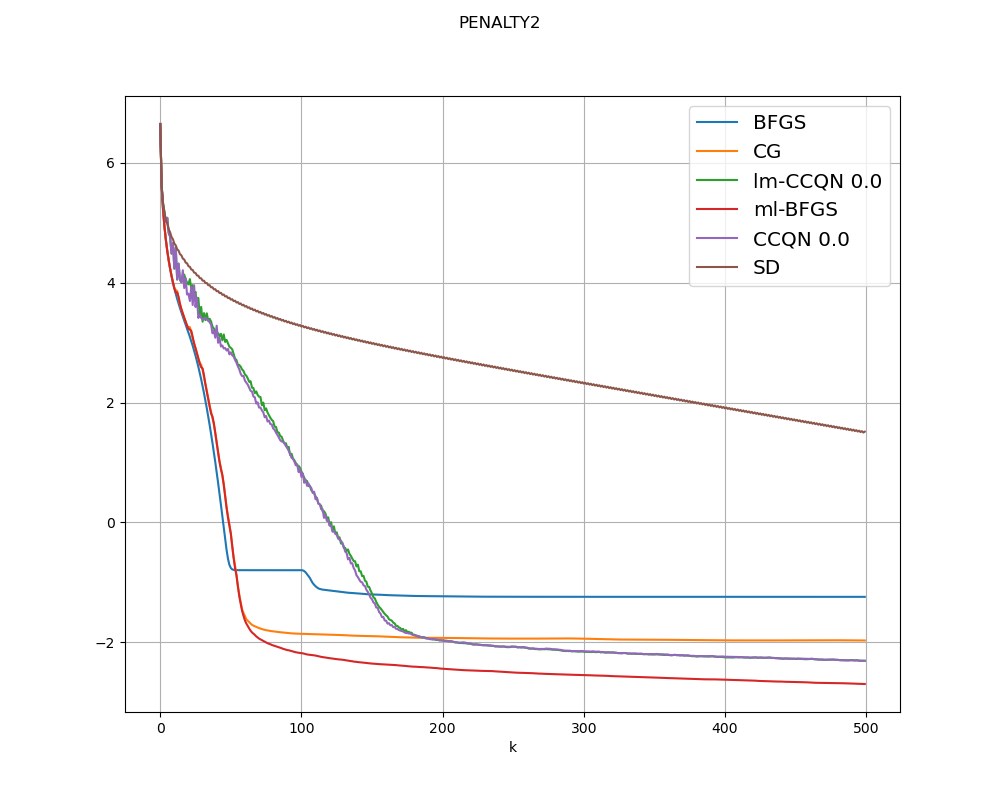}}}%
                \qquad
                \subfloat[\texttt{SROSENBR}]{{\includegraphics[width=6.5cm]{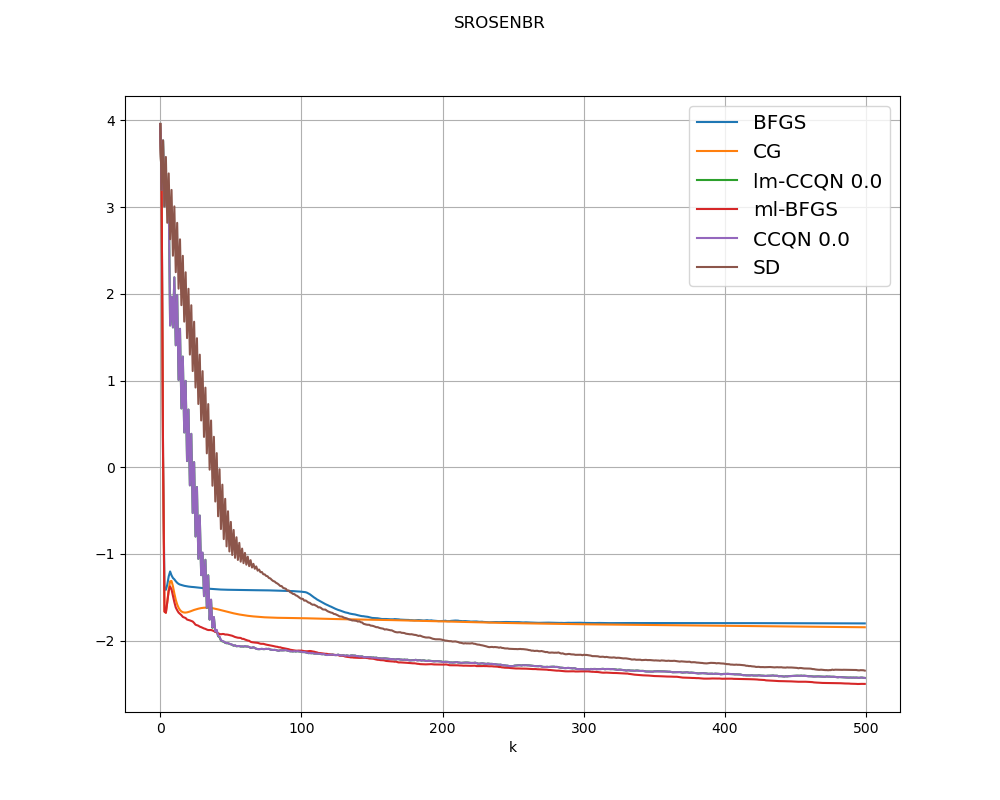}}}%
                \caption{Average log norm of the gradient at step $k$ for each tested method.}
\end{figure}

\begin{figure}[H]
                \centering
                \subfloat[\texttt{TQUARTIC}]{{\includegraphics[width=6.5cm]{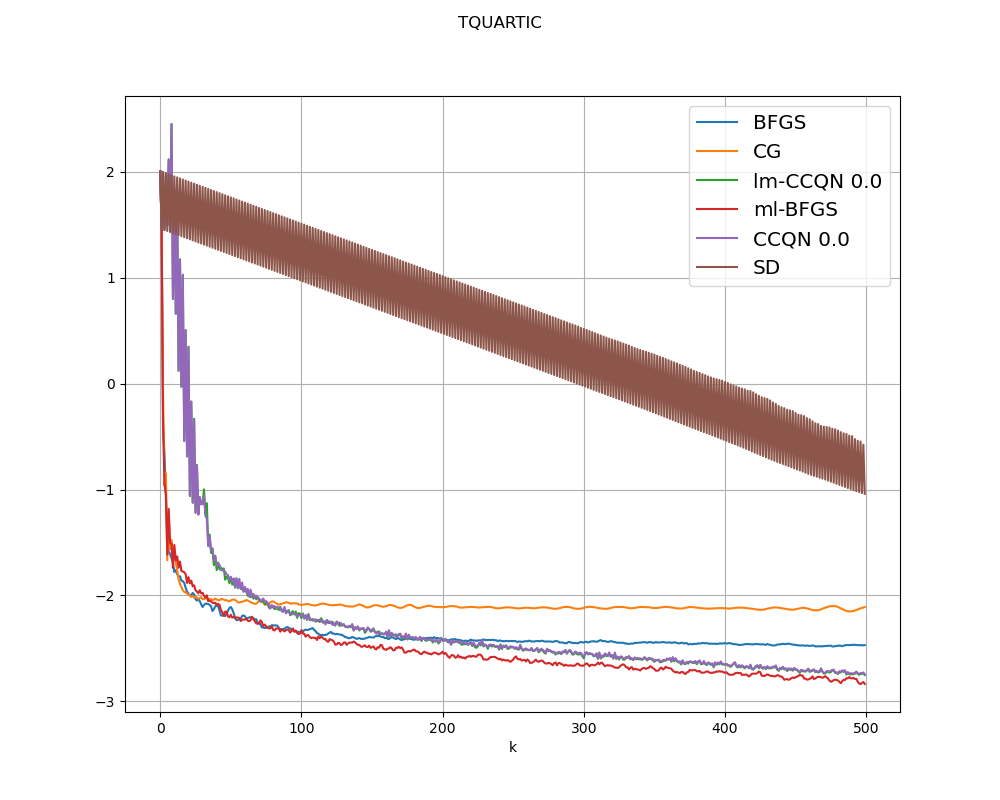}}}%
                \qquad
                \subfloat[\texttt{WOODS}]{{\includegraphics[width=6.5cm]{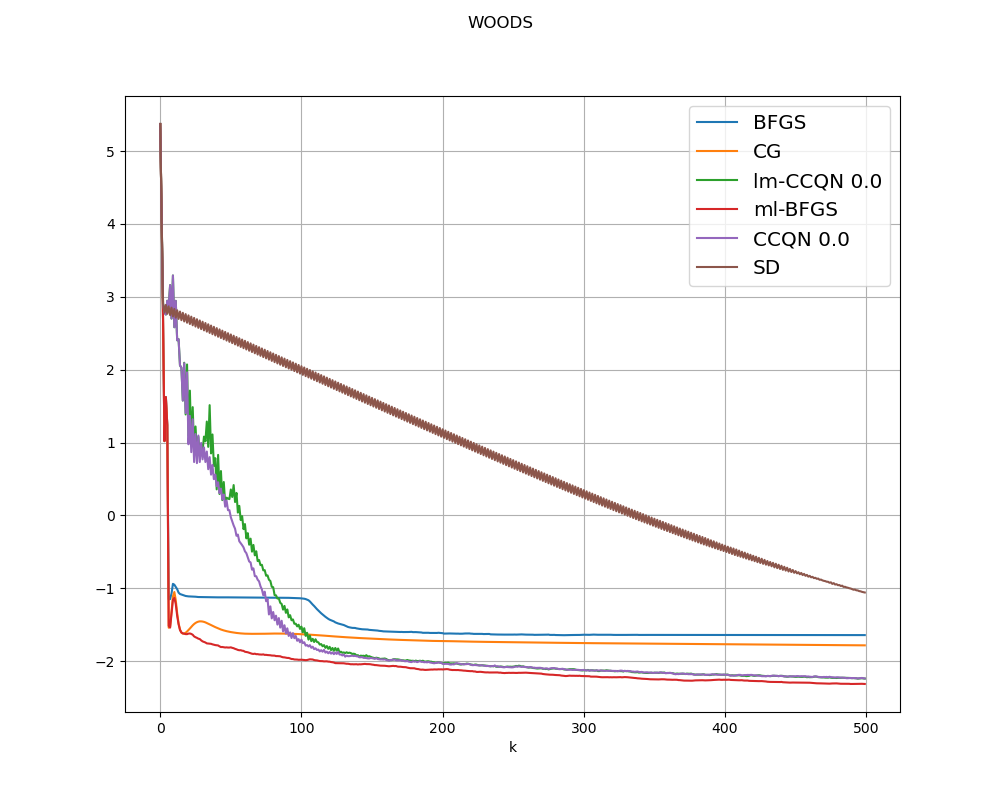}}}%
                \caption{Average log norm of the gradient at step $k$ for each tested method.}
\end{figure}

\begin{figure}[H]
                \centering
                \subfloat[\texttt{INTEQNELS}]{{\includegraphics[width=6.5cm]{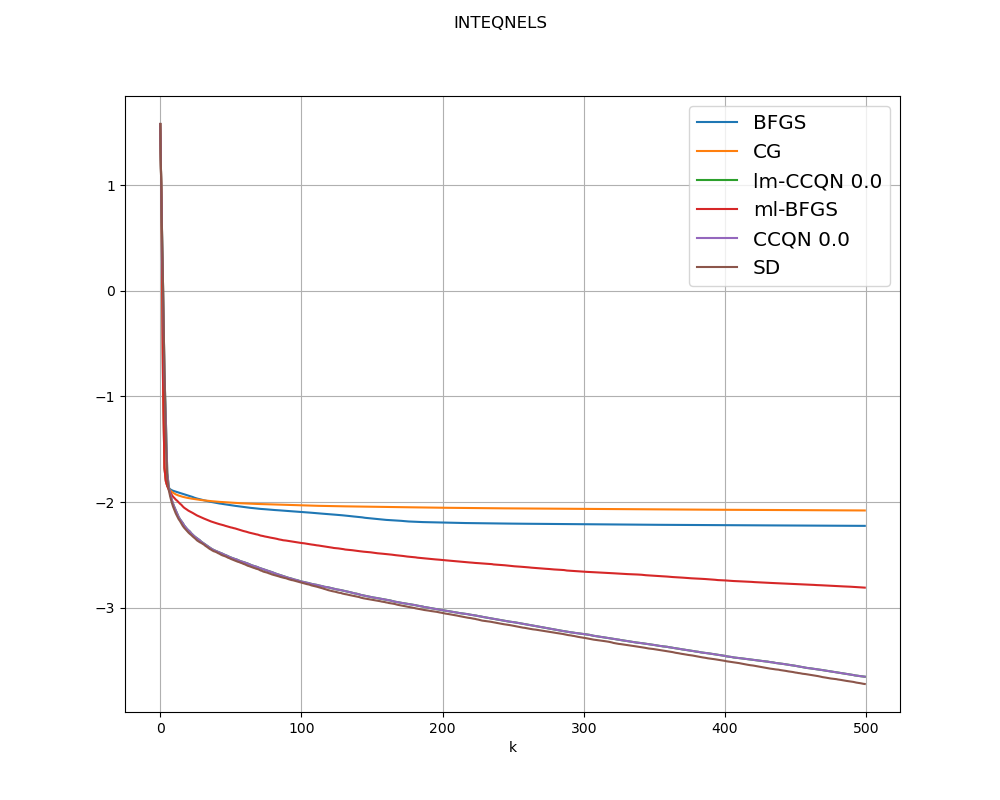}}}%
                \qquad
                \subfloat[\texttt{CHNROSNB}]{{\includegraphics[width=6.5cm]{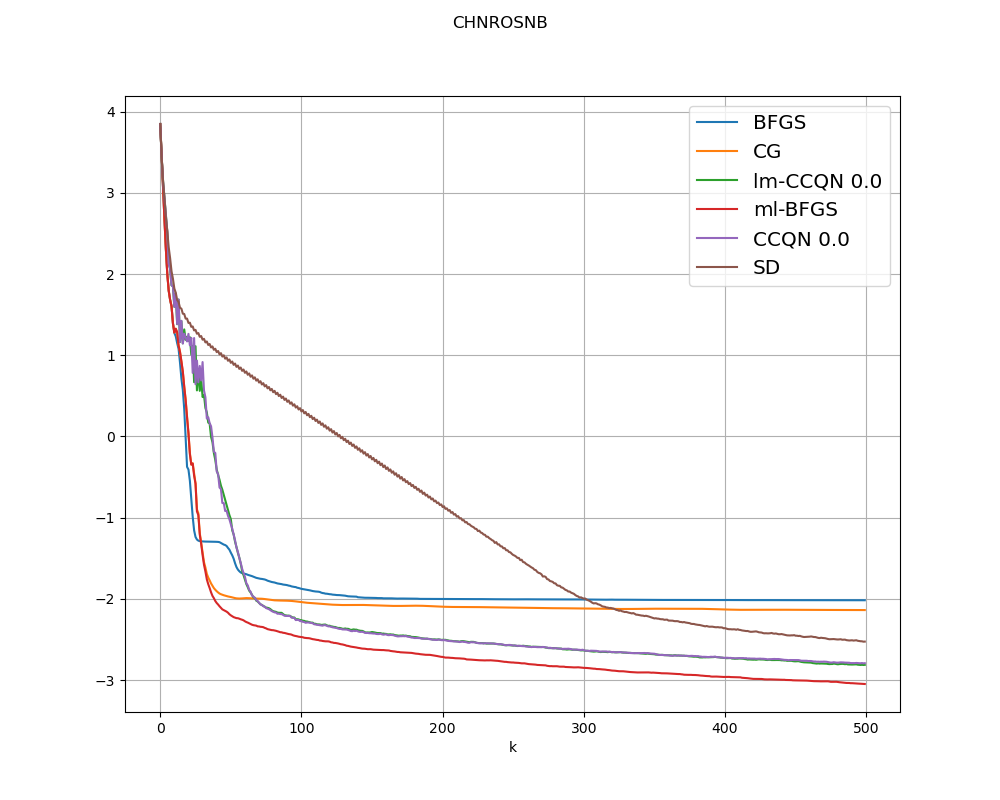}}}%
                \caption{Average log norm of the gradient at step $k$ for each tested method.}
\end{figure}

\begin{figure}[H]
                \centering
                \subfloat[\texttt{ERRINROS}]{{\includegraphics[width=6.5cm]{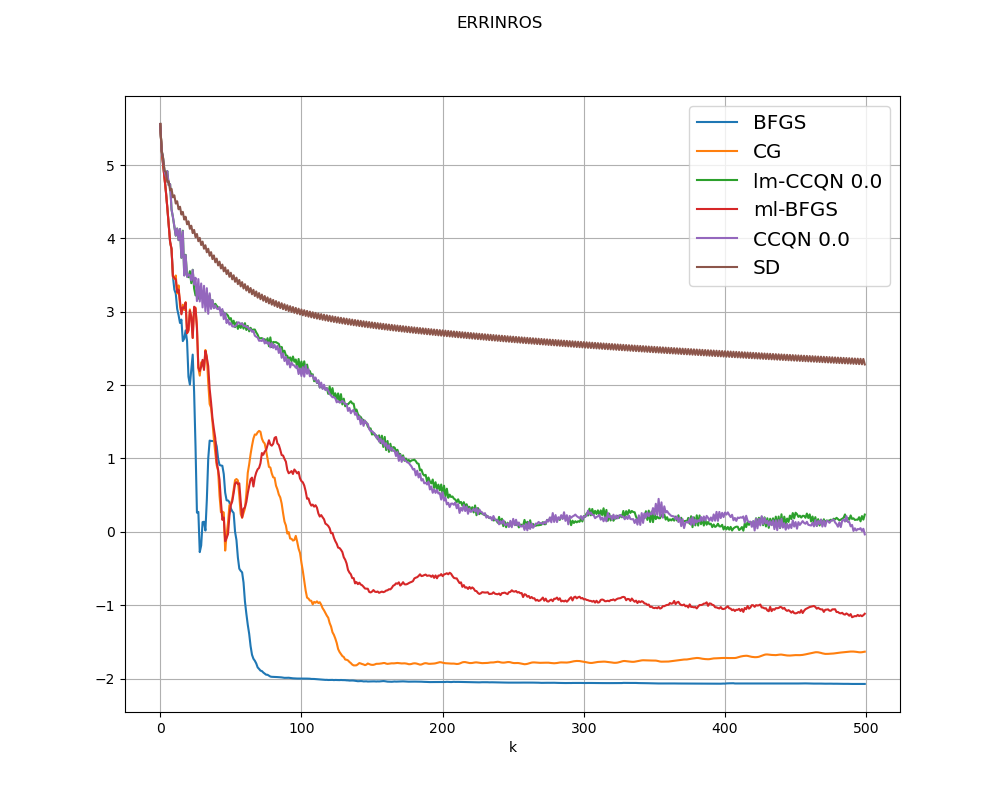}}}%
                \qquad
                \subfloat[\texttt{ERRINRSM}]{{\includegraphics[width=6.5cm]{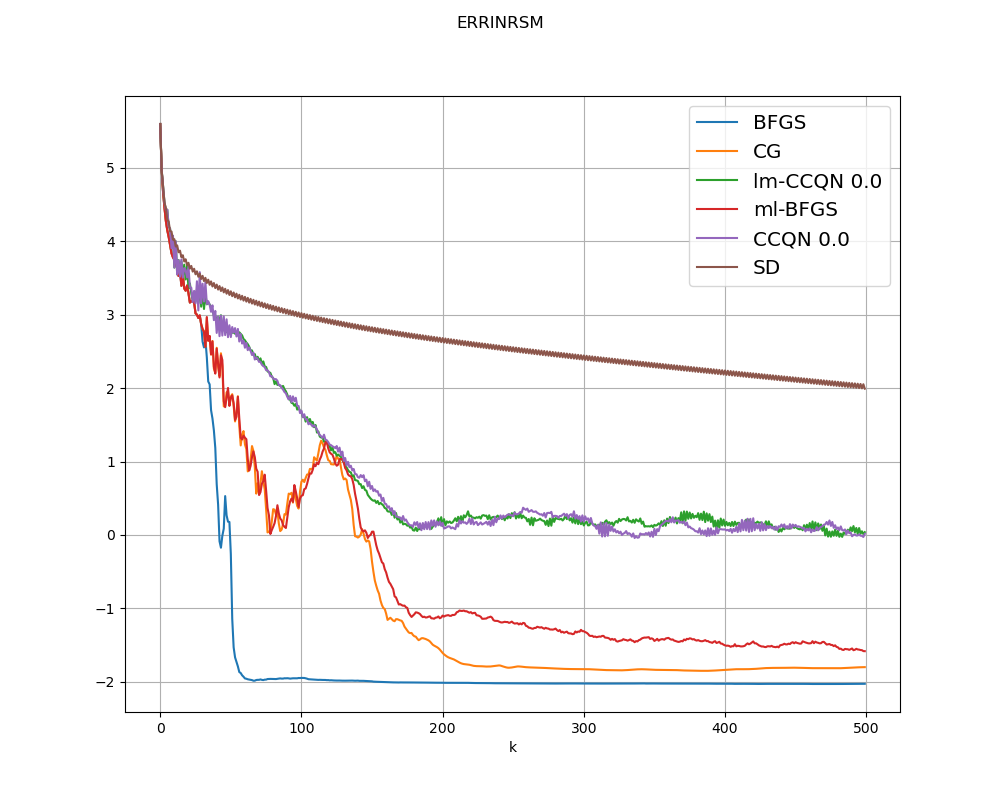}}}%
                \caption{Average log norm of the gradient at step $k$ for each tested method.}
\end{figure}

\begin{figure}[H]
                \centering
                \subfloat[\texttt{DIXON3DQ\_n1000}]{{\includegraphics[width=6.5cm]{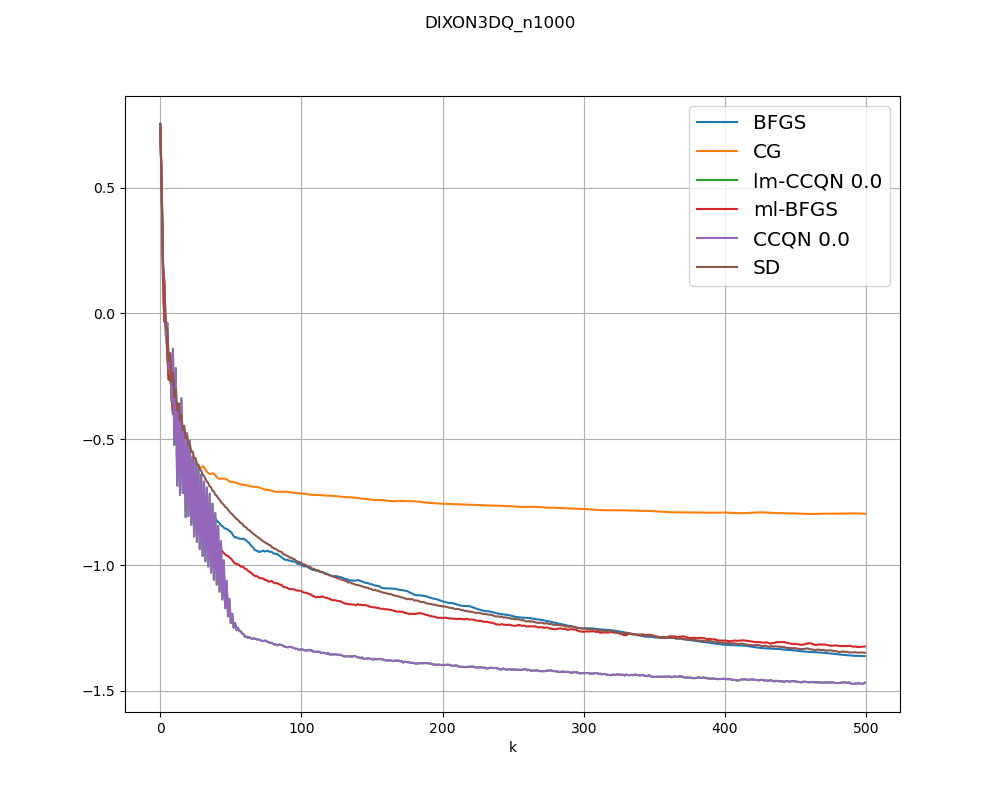}}}%
                \qquad
                \subfloat[\texttt{DIXON3DQ\_n100}]{{\includegraphics[width=6.5cm]{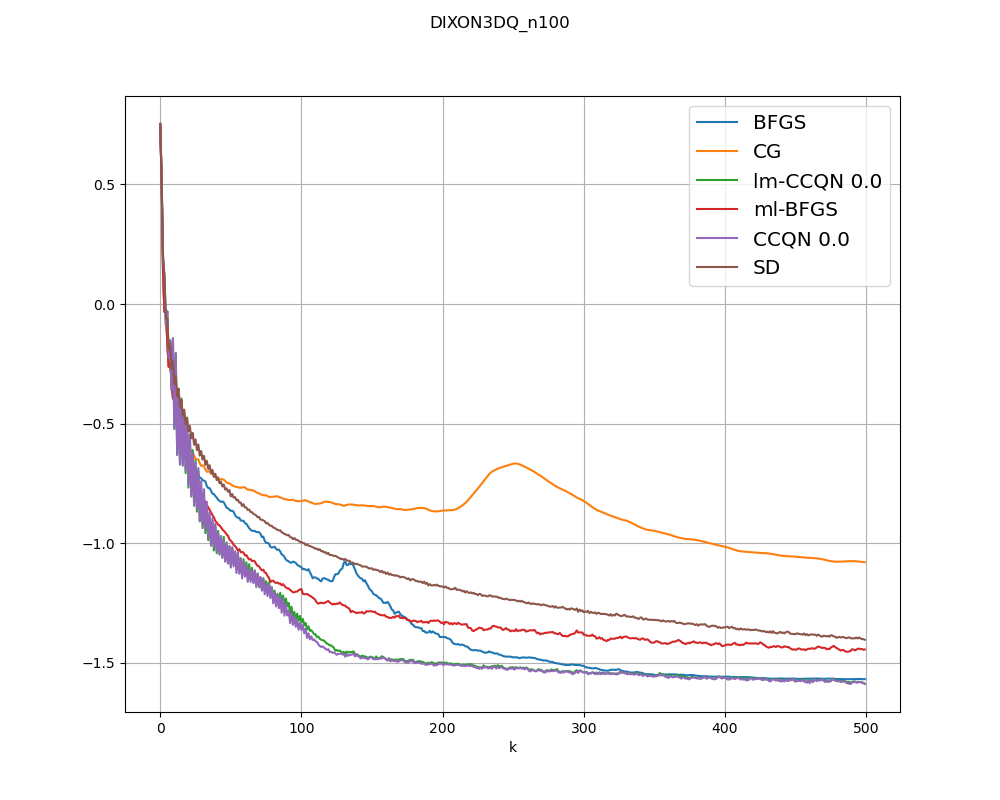}}}%
                \caption{Average log norm of the gradient at step $k$ for each tested method.}
\end{figure}

\begin{figure}[H]
                \centering
                \subfloat[\texttt{DIXON3DQ}]{{\includegraphics[width=6.5cm]{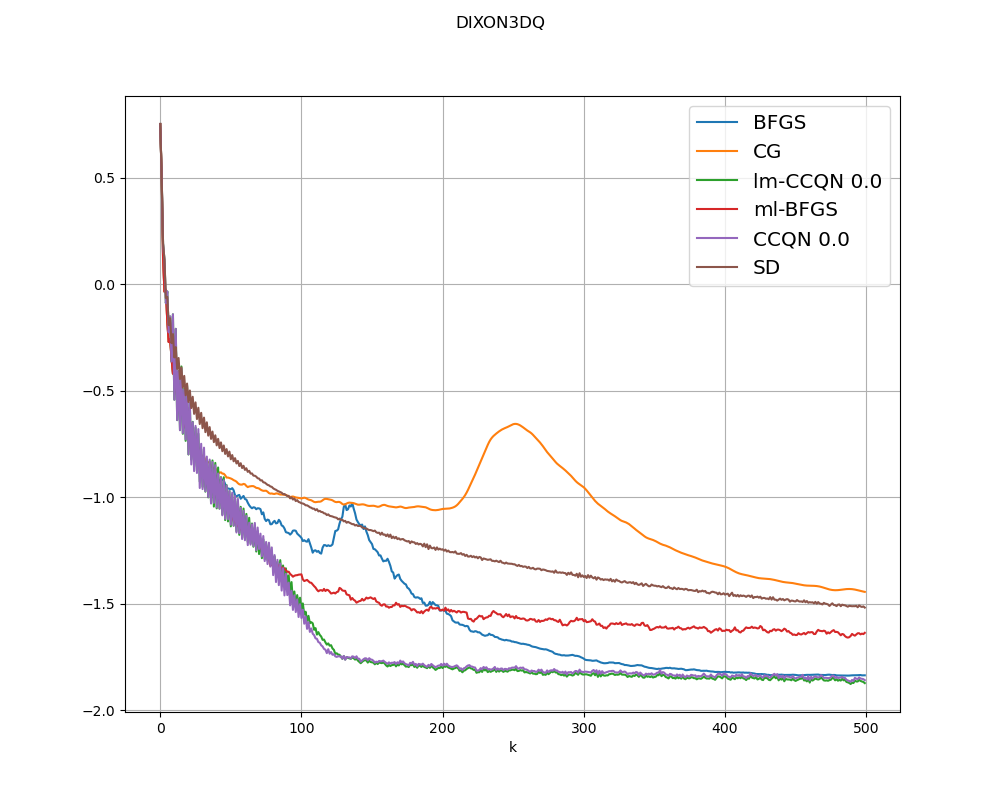}}}%
                \qquad
                \subfloat[\texttt{DQDRTIC\_n100}]{{\includegraphics[width=6.5cm]{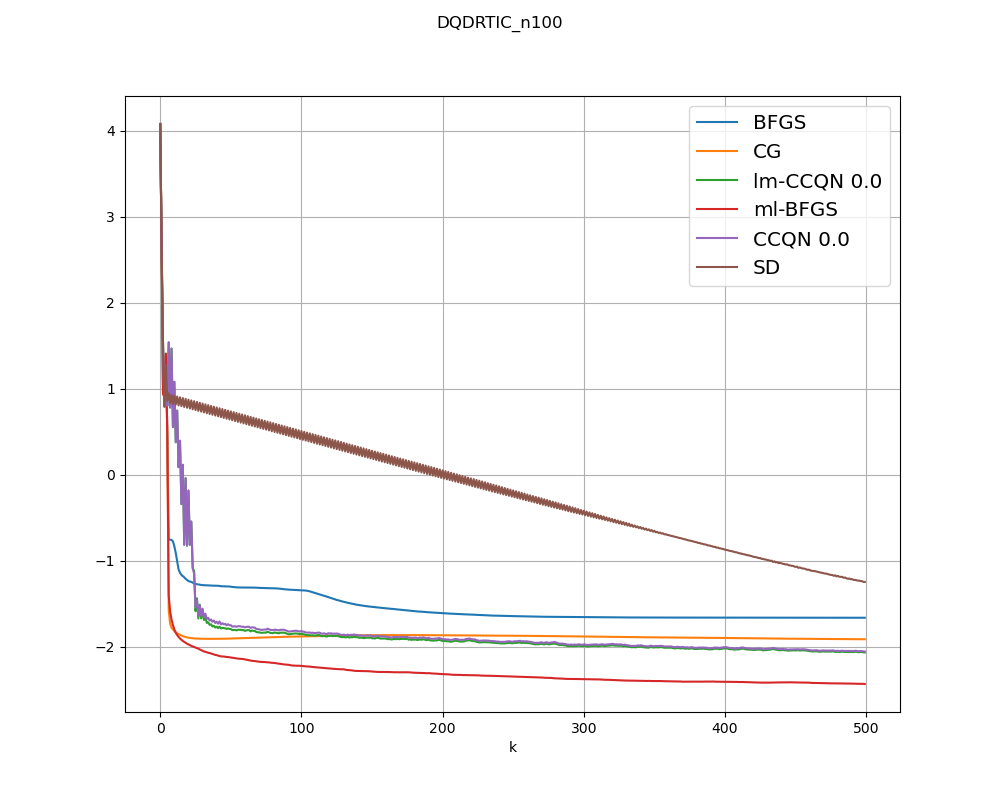}}}%
                \caption{Average log norm of the gradient at step $k$ for each tested method.}
\end{figure}

\begin{figure}[H]
                \centering
                \subfloat[\texttt{DIXON3DQ}]{{\includegraphics[width=6.5cm]{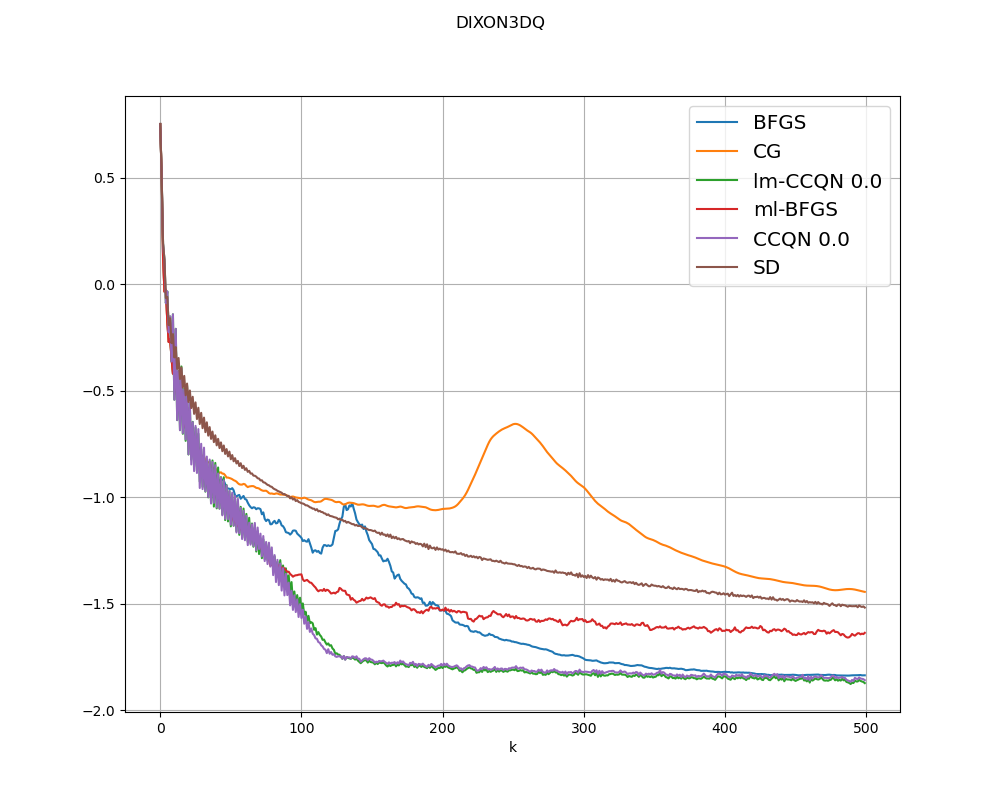}}}%
                \qquad
                \subfloat[\texttt{DQDRTIC\_n100}]{{\includegraphics[width=6.5cm]{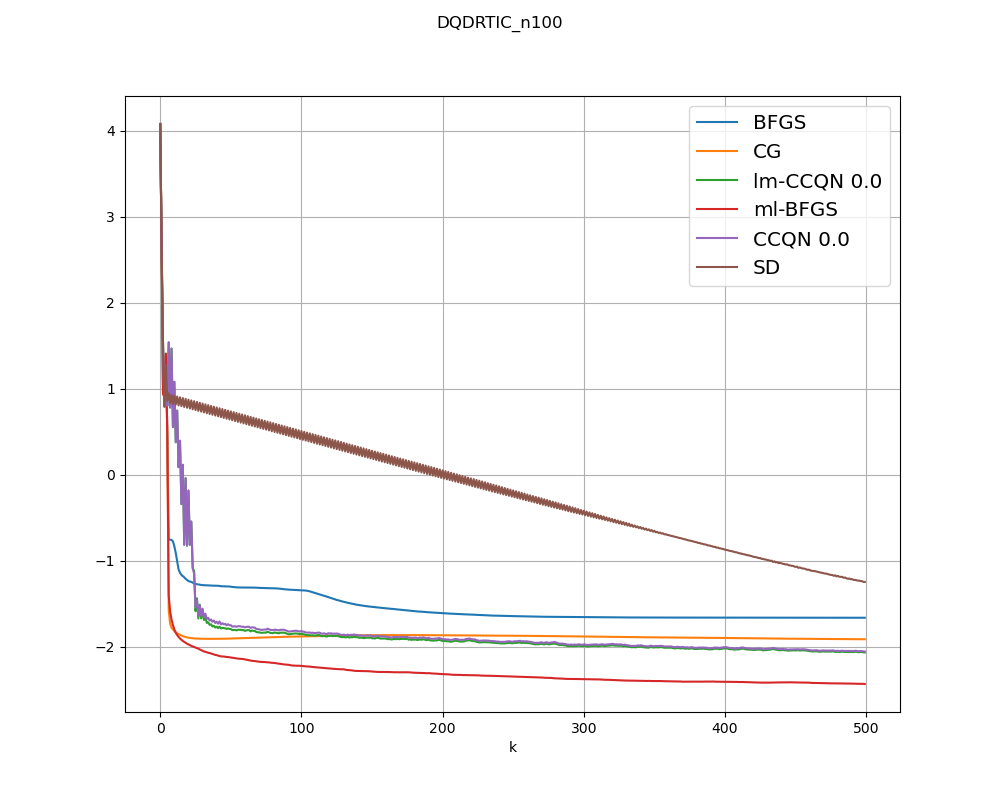}}}%
                \caption{Average log norm of the gradient at step $k$ for each tested method.}
\end{figure}

\begin{figure}[H]
                \centering
                \subfloat[\texttt{TRIDIA\_n100}]{{\includegraphics[width=6.5cm]{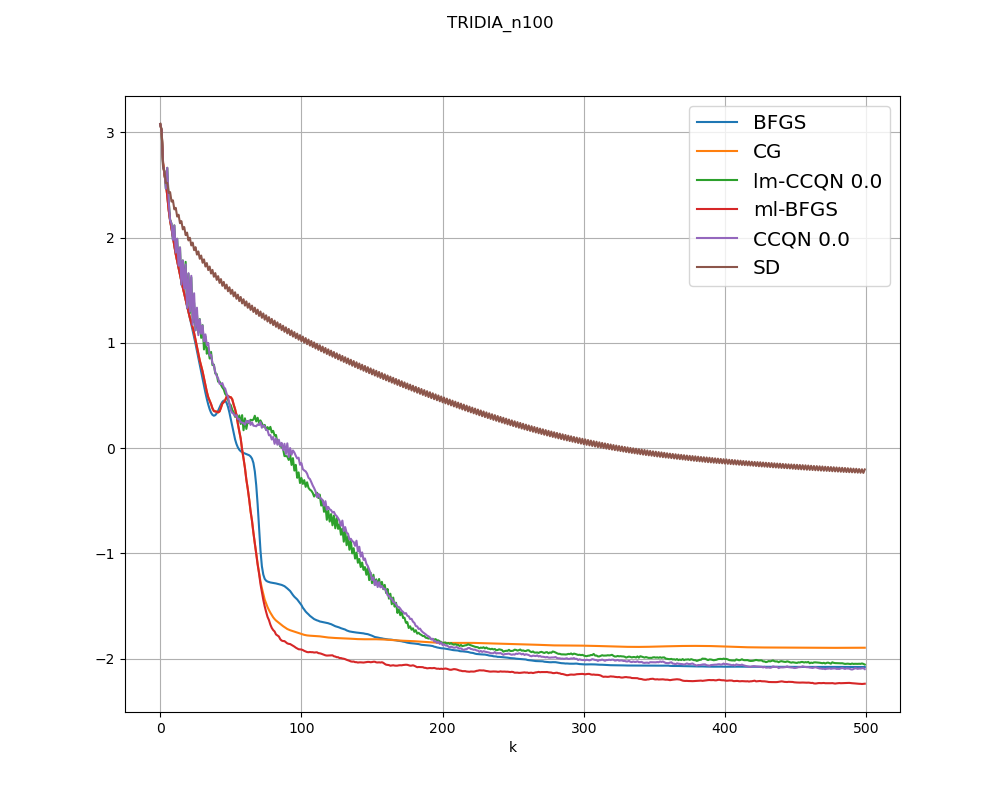}}}%
                \qquad
                \subfloat[\texttt{TESTQUAD\_n1000}]{{\includegraphics[width=6.5cm]{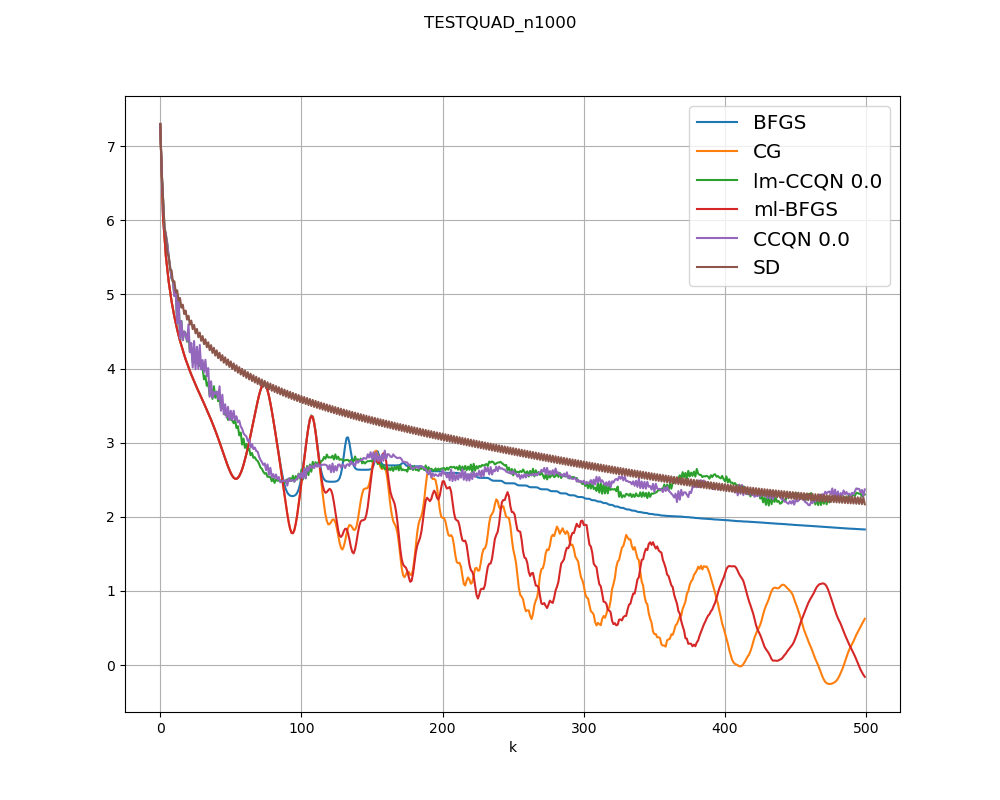}}}%
                \caption{Average log norm of the gradient at step $k$ for each tested method.}
\end{figure}

\begin{figure}[H]
                \centering
                \subfloat[\texttt{HILBERTB\_n50}]{{\includegraphics[width=6.5cm]{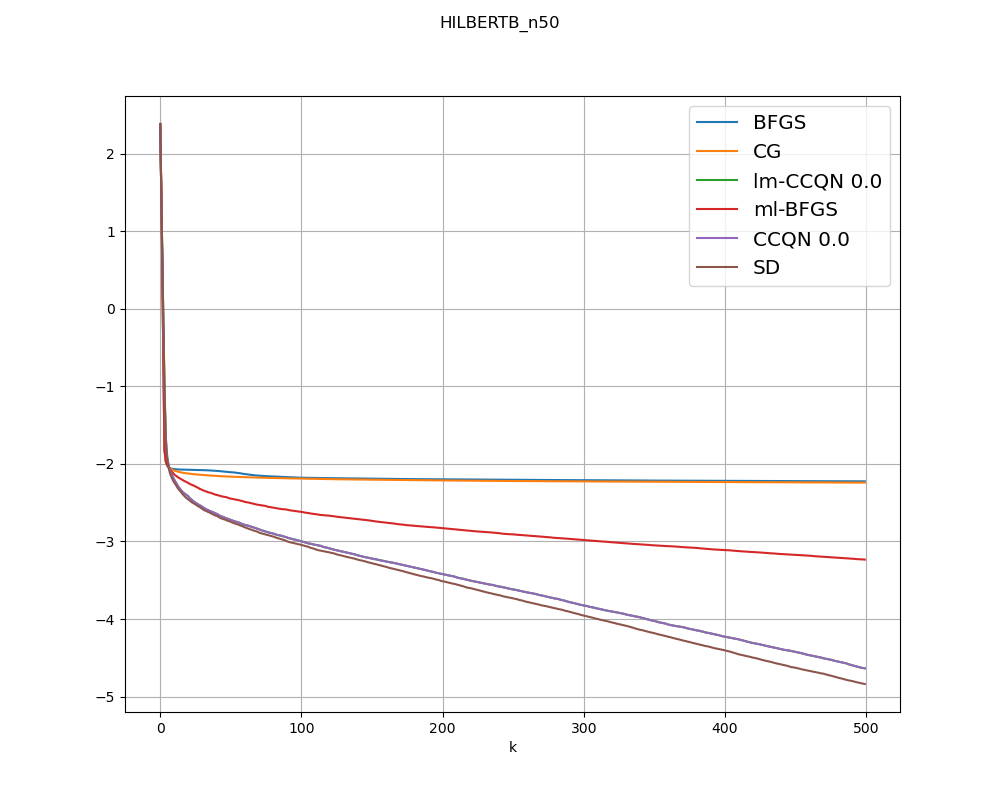}}}%
                \qquad
                \subfloat[\texttt{TOINTQOR\_n50}]{{\includegraphics[width=6.5cm]{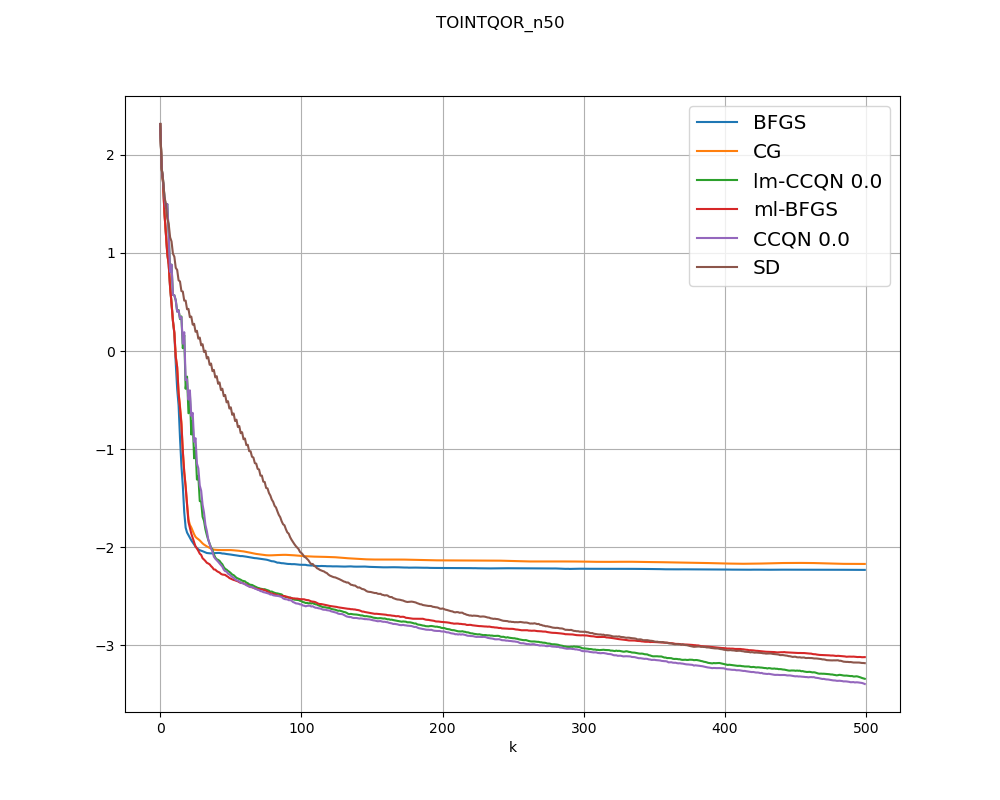}}}%
                \caption{Average log norm of the gradient at step $k$ for each tested method.}
\end{figure}


\end{document}